\documentclass[11pt,a4paper]{article}
\usepackage{a4wide}
\usepackage{amssymb,amsfonts,amsmath,amsthm}
\usepackage{subcaption}
\usepackage[]{graphicx}
\usepackage[hidelinks]{hyperref}
\usepackage{doi}
\usepackage{mathtools}
\usepackage{xparse}

\usepackage{amssymb,amsfonts,amsmath}
\usepackage{graphicx,color,changebar}
\usepackage{autonum}

\DeclareFontFamily{U}{mathb}{\hyphenchar\font45}
\DeclareFontShape{U}{mathb}{m}{n}{
<-6> mathb6 <6-7> mathb6 <7-8> mathb6
<8-9> mathb6 <9-10> mathb6
<10-12> mathb6 <12-> mathb6
}{}
\DeclareSymbolFont{mathb}{U}{mathb}{m}{n}
\DeclareMathSymbol{\llcurly}{\mathrel}{mathb}{"CE}
\DeclareMathSymbol{\ggcurly}{\mathrel}{mathb}{"CF}
\newcommand{\eg}{e.\,g.,\ }

\newcommand{\A}{\mathcal{A}}

\newcommand{\E}{\mathcal{E}}
\newcommand{\Hn}{ {\mathbb{H}_n} }   

\newcommand{\XWpdc}{ {\mathbb{X}_c^{\raisebox{0.2em}{{\fontsize{3}{2}\selectfont $\succ $}}}} }  
\newcommand{\XWpdpdc}{ {\mathbb{X}_c^{\raisebox{0.2em}{{\fontsize{5}{2}\selectfont $\ggcurly$}}}} }  
\newcommand{\XWpdd}{ {\mathbb{X}_d^{\raisebox{0.2em}{{\fontsize{3}{2}\selectfont $\succ $}}}} }  
\newcommand{\XWpdpd}{ {\mathbb{X}_d^{\raisebox{0.2em}{{\fontsize{5}{2}\selectfont $\ggcurly$}}}} }  
\newcommand{\XWpdpdd}{ {\mathbb{X}_d^{\raisebox{0.2em}{{\fontsize{5}{2}\selectfont $\ggcurly$}}}} }

\newtheorem{remark}{Remark}[section]

\newtheorem{lemma}{Lemma}[section]

\newtheorem{corollary}{Corollary}[section]
\newtheorem{example}{Example}[section]
\newcommand{\M}{{\mathcal M}}

\newcommand{\R}{\mathbb{R}}
\DeclareMathOperator{\vect}{vec}
\DeclareMathOperator{\tr}{tr}
\DeclareMathOperator{\adj}{adj}
\newcommand{\Ricc}{\mathsf{Ricc}}

\DeclareMathOperator{\rank}{rank}
\newcommand {\matr}      [1] {\left[\begin{array}{#1}}
\newcommand {\rix}          {\end{array}\right]}

\newcommand{\ie}{i.\,e.\ }
\newcommand{\imag}{\imath}
\usepackage{xstring}
\newcommand{\mat}[3][C]{
	\mathbb{#1}^{
		\IfSubStr{#2}{+}{(#2)}{#2}
		\times
		\IfSubStr{#3}{+}{(#3)}{#3}
	}
}
\DeclareDocumentCommand{\matz}{m m O{K} O{z}}{\mathbb{#3}[{#4}]^
	{
		\IfSubStr{#1}{+}{(#1)}{#1}
		\times
		\IfSubStr{#2}{+}{(#2)}{#2}
}}
\newcommand{\diffX}[1][X]{\frac{\partial}{\partial #1}}
\DeclareDocumentCommand{\sp}{m m}{
	\left \langle #1,\,#2\right\rangle
}
\DeclareMathOperator*{\argmax}{\arg\,\max}

\title{Computation of the analytic center of the solution set of the linear matrix inequality arising in continuous- and discrete-time passivity analysis}                                       

\author{D. Bankmann\footnotemark[1], V. Mehrmann\footnotemark[2], Y. Nesterov\footnotemark[3], and P. Van Dooren\footnotemark[4]
}
\date{\today}
\begin{document}
\maketitle

\begin{abstract} In this paper formulas are derived for the analytic center of the solution set of linear matrix inequalities (LMIs) defining passive transfer functions. The algebraic Riccati equations that are usually associated with such systems are related to boundary points of the convex set defined by the solution set of the LMI\@. It is shown that the analytic center is  described by closely related matrix equations, and  their properties are analyzed for continuous- and discrete-time systems. Numerical methods are derived to solve these equations via steepest ascent and Newton-like methods. It is  also shown that the analytic center has nice robustness properties
	when it is used to represent passive systems. The results are illustrated by numerical examples.
\end{abstract}

\footnotetext[1]{
	Institut f\"ur Mathematik MA 4-5, TU Berlin, Str.\ des 17.\ Juni 136,
	D-10623 Berlin, FRG.
\url{bankmann@math.tu-berlin.de}. Supported by the German Research Foundation DFG as part of the project `Distributed Dynamic Security Control in Next-Generation Electrical Power Systems' with the project identification number 361092219 of the priority program `DFG SPP 1984 - Hybrid and multimodal energy systems: System theory methods for the transformation and operation of complex networks'. }
\footnotetext[2]{
	Institut f\"ur Mathematik MA 4-5, TU Berlin, Str.\ des 17.\ Juni 136,
	D-10623 Berlin, FRG.
	\url{mehrmann@math.tu-berlin.de}. Supported by {\it the German Federal Ministry of Education and Research BMBF within the project EiFer}
	and by {\it Deutsche Forschungsgemeinschaft},
through TRR 154 'Mathematical Modelling, Simulation and Optimization using the Example of Gas Networks'.}
\footnotetext[3]{
	Department of Mathematical Engineering, Universit\'e catholique de Louvain, Louvain-La-Neuve, Belgium.
\url{yurii.nesterov@uclouvain.be}. Supported by ERC Advanced  Grant 788368.}
\footnotetext[4]{
	Department of Mathematical Engineering, Universit\'e catholique de Louvain, Louvain-La-Neuve, Belgium.
	\url{paul.vandooren@uclouvain.be}. Supported  by {\it Deutsche Forschungsgemeinschaft},
	through TRR 154 'Mathematical Modelling, Simulation and Optimization using the Example of Gas Networks'.
}

\section{Introduction}
We consider realizations of linear dynamical systems that are denoted as \emph{positive real or passive} and their associated transfer functions. In particular, we study positive  transfer functions  which play a fundamental role in systems and
control theory: they represent \eg spectral density functions of
stochastic processes, show up in spectral factorizations, are the Hermitian part of a
positive real transfer function, characterize port-Hamiltonian systems, and are also
related to algebraic Riccati equations.

Positive transfer functions form a convex set, and this property has lead to the extensive use of convex optimization techniques
in this area (especially for so-called linear matrix inequalities
\cite{BoyEFB94}).  In order to optimize a certain scalar function
$f(X)$ over a convex set, one often defines a barrier function $b(X)$ that becomes infinite near
the boundary of the set, and then finds the minimum of $c \cdot f(X)+b(X)$,
$c \succeq 0$, as $c \rightarrow +\infty$. These minima (which are functions of the
parameter $c$) are called the points of the {\em central path}.
The starting point of this path ($c = 0$) is called the {\em analytic center} \/of the set.

In this paper we present an explicit set of equations that define the analytic center of the solution set of the linear matrix inequality defining a passive transfer
function. We also show how these equations relate to the algebraic
Riccati equations that typically arise in the spectral factorization of
transfer functions.
We discuss transfer functions both on the imaginary axis (\ie the continuous-time case), as well as on the unit circle (\ie the discrete-time case).
In the continuous-time setting the transfer function arises from the \emph{Laplace transform} of the  system
\begin{equation} \label{statespace_c}
	\begin{array}{rcl} \dot x & = & Ax + B u,\ x(0)=0,\\
		y&=& Cx+Du,
	\end{array}
	\end{equation}
	where $u:\mathbb R\to\mathbb{C}^m$,   $x:\mathbb R\to \mathbb{C}^n$,  and  $y:\mathbb R\to\mathbb{C}^m$  are vector-valued functions denoting, respectively, the \emph{input}, \emph{state},
	and \emph{output} of the system. Denoting real and complex $n$-vectors ($n\times m$ matrices) by $\mathbb R^n$, $\mathbb C^{n}$ ($\mathbb R^{n \times m}$, $\mathbb{C}^{n \times m}$), respectively, the coefficient matrices satisfy $A\in \mathbb{C}^{n \times n}$,   $B\in \mathbb{C}^{n \times m}$, $C\in \mathbb{C}^{m \times n}$, and  $D\in \mathbb{C}^{m \times m}$.

	In the discrete-time setting the transfer function arises from the \emph{z-transform} applied to the system
	\begin{equation} \label{statespace_d}
		\begin{array}{rcl} x_{k+1} & = & Ax_k + B u_k,\ x_0=0,\\
			y_k&=& Cx_k+Du_k,
		\end{array}
		\end{equation}
		with state, input, and output sequences $\{x_k\}$, $\{u_k\}$, $\{y_k\}$.
		In both cases, we usually denote these systems by four-tuples of matrices $\M:=\left\{A,B,C,D\right\}$ and the associated transfer functions by
		\begin{equation} \label{ABCD}
			\mathcal T_c(s):=D+C(sI_n-A)^{-1}B, \qquad \mathcal T_d(z):=D+C(z I_n-A)^{-1}B,
		\end{equation}
		respectively.

		We restrict ourselves to systems which are \emph{minimal}, \ie the pair $(A,B)$ is \emph{controllable}  (for all $\lambda\in \mathbb C$, $\rank \mbox{\small $[\,\lambda I-A \ B\,]$} =n$), and the pair $(A,C)$ is \emph{reconstructable} (\ie $(A^\mathsf{H},C^\mathsf{H})$ is controllable). Here, the conjugate transpose (transpose) of a vector or matrix $V$ is denoted by
		$V^{\mathsf{H}}$ ($V^{\mathsf{T}}$) and the identity matrix is denoted by $I_n$ or $I$ if the dimension is clear. We furthermore require that input and output port dimensions are equal  to $m$ and assume that $\rank B=\rank C=m$.

		\emph{Passive} systems and their relationships with \emph{positive-real transfer functions} are well studied, starting with the works  \cite{Kal63,Pop73,Wil71,Wil72a,Wil72b,Yak62} and the topic has recently received a revival in the work on \emph{port-Hamiltonian (pH) systems}, \cite{Sch04,SchJ14}. For a summary of the relationships see \cite{BeaMV19, Wil71}, where also the characterization of  passivity via the solution set of an associated \emph{linear matrix inequality (LMI)} is highlighted.

		The paper is organized as follows. After some preliminaries in Section~\ref{sec:prelim}, in Section~\ref{sec:ana_center} we study the analytic centers of the solution sets of LMIs associated with the continuous- and discrete-time
		case. In Section~\ref{sec:numerical} we discuss numerical methods to compute the analytic centers  using steepest ascent
		as well as  Newton-like methods and show that the analytic centers can be computed efficiently. In Section~\ref{sec:passivity} lower bounds for the distance to non-passivity (the passivity radius) are derived using smallest eigenvalues of the
		Hermitian matrices associated with the linear matrix inequalities evaluated at the analytic center. The results are illustrated with some simple examples where the analytic center can be calculated analytically.
		In Appendix~\ref{der_complex_functions} we derive formulas for the computation of the gradients and the Hessian of the functions that we optimize and in Appendix~\ref{appendix:cayley} we clarify some of the differences that arise between the continuous- and the discrete-time case.

		\section{Preliminaries}\label{sec:prelim}
		Throughout this article we will use the following notation.
		We denote the set of Hermitian matrices in $\mathbb{C}^{n \times n}$ by $\Hn$.
		Positive definiteness (semidefiniteness) of  $A\in \Hn$ is denoted by $A\succ 0$ ($A\succeq 0$).
		The real and imaginary parts of a complex matrix $Z$ are written as $\Re (Z)$ and $\Im (Z)$, respectively, and $\imath$ is the imaginary unit.
		%
		%
		We consider functions over $\Hn$,
		which is a vector space if considered as a \emph{real} subspace of $\mat[R]nn+\imath \mat[R]nn$.
		We will identify $\mat[C]mn$ with $\mat[R]mn+\imath \mat[R]mn $, but we note that this has implications when one is carrying out differentiations, see  Appendix~\ref{der_complex_functions}. 
		The \emph{Frobenius scalar product}  for matrices $X, Y\in\mat[R]nn+\imath \mat[R]nn $ is given by
		\begin{equation}
			\sp{X}{Y}_\R:= \Re(\tr(A^{\mathsf{H}}B)) = \tr(Y_r^TX_r+Y_i^TX_i),
		\end{equation}
		where we have partitioned $X,Y$ as $X=X_r+\imath X_i$, $Y=Y_r+\imath Y_i$ with real and imaginary parts in $\mathbb{R}^{n\times n}$.
		As we are mainly concerned with this scalar product, we will drop the subscript $\R$.
		%
		We will make frequent use of the following properties of this inner product given by
		\begin{equation}
			\label{eq:propprod} \sp{X}{Y}=\sp{Y}{X}, \, \|X\|_F=\sp{X}{X}^{\frac{1}{2}}, \, \sp{X}{YZ}=\sp{Y^{\mathsf{H}}X}{Z}=\sp{XZ^{\mathsf{H}}}{Y}.
		\end{equation}

		The concepts of \emph{positive-realness} and \emph{passivity} are well studied. In the following subsections we briefly recall some important properties following \cite{GenNV99,Wil71}, where we repeat a few observations from \cite{BeaMV19}. See also \cite{Wil71} for a more detailed survey.

		\subsection{Positive-realness and passivity, continuous-time}

		Consider a continuous-time system $\M $ as in (\ref{statespace_c})  and the transfer function $\mathcal T_c$ as in \eqref{ABCD}. The transfer function $\mathcal T_c(s)$ is called {\em positive real} \/if the 
		matrix-valued rational function
		\[ 
			\Phi_c(s):= \mathcal T_c^{\mathsf{H}}(-s) + \mathcal T_c(s)
		\] 
		is positive semidefinite for $s$ on the imaginary axis, \ie $\Phi_c(\imath\omega)\succeq 0$
		for all $\omega\in \mathbb{R}$ and it is called \emph{strictly positive real} if $\Phi_c(\imath \omega)\succ 0$ for all $\ \omega\in \mathbb{R}$.

		{We associate with $\Phi_c$ a system pencil
		\begin{equation} \label{statespace}
			S_{c}(s) :=
			\left[ \begin{array}{cc|c} 0 & A-sI_n & B \\
			A^{\mathsf{H}}+sI_n & 0 & C^{\mathsf{H}} \\ \hline B^{\mathsf{H}} & C & R  \end{array} \right],
		\end{equation}
		where $R:=D+D^{\mathsf{H}}$. Here \eqref{statespace} has a Schur complement which is the transfer function
		$\Phi_c(s)$ and the generalized eigenvalues of $S_{c}(s)$ are the zeros of $\Phi_c(s)$.

		For  $X \in \Hn$ we introduce the matrix function
		\begin{equation} \label{prls}
			W_c(X) := \left[
				\begin{array}{cc}
					-X\,A - A^{\mathsf{H}}X & C^{\mathsf{H}} - X\,B \\
					C- B^{\mathsf{H}}X & D+D^{\mathsf{H}}
				\end{array}
			\right],
		\end{equation}
		%
		%
		%
		%
		%
		%
		%
		If $\mathcal T_c(s)$ is  positive real, then
the linear matrix inequality (LMI)
		\begin{equation} \label{KYP-LMI}
			W_c(X) \succeq 0
		\end{equation}
has a solution 	$X\in \Hn$
and we have the sets}
		\begin{subequations}\label{LMIsolnsets}
			\begin{align}
&\XWpdc:=\left\{ X\in \Hn \left|   W_c(X) \succeq 0,\ X \succ 0 \right.\right\}, 
\label{XpdsolnWpsd} \\[1mm]
&\XWpdpdc :=\left\{ X\in \Hn \left|   W_c(X) \succ  0,\ X \succ 0 \right.\right\}. \label{XpdsolWpd}
\end{align}
			\end{subequations}
			%
			An important subset of $\XWpdc$ are those solutions to (\ref{KYP-LMI}) for which the
			rank $r$ of $W_c(X)$ is minimal ({\ie} for which $r=\rank\Phi_c(s)$).
			If $R$ is invertible, then
			the minimum rank solutions in $\XWpdc$
			are those for which $\rank W_c(X) = \rank (R)  = m$, which in turn is the case
			if and only if the Schur complement of $R$ in $W_c(X)$ is zero.  This Schur
			complement is associated with the continuous-time \emph{algebraic Riccati equation (ARE)}
			\begin{equation}
				\mathsf{Ricc}_c(X) := -XA-A^{\mathsf{H}}X  -(C^{\mathsf{H}}-XB)R^{-1}(C-B^{\mathsf{H}}X)=0.\label{riccatic}
			\end{equation}
			Solutions $X$ to (\ref{riccatic}) produce a spectral factorization of $\Phi_c(s)$, and each solution corresponds to a
			\emph{Lagrangian invariant subspace} spanned by the columns of $U_c:=\matr{cc} I_n & -X^{\mathsf{T}} \rix^{\mathsf{T}} $
			that remains invariant under the action of the \emph{Hamiltonian matrix}
			\begin{equation}\label{HamMatrix}
				{\mathcal H}_c:=\matr{cc} A-B R^{-1} C & - B R^{-1} B^{\mathsf{H}} \\
				C^{\mathsf{H}} R^{-1} C & -(A-B R^{-1} C)^{\mathsf{H}} \rix,
			\end{equation}
			\ie $U_c$ satisfies ${ \mathcal H}_cU_c=U_c A_{F_c}$ for a \emph{closed loop matrix} $A_{F_c}=A-BF_c$ with $F_c := R^{-1}(C-B^{\mathsf{H}}X)$ (see e.g., \cite{FreMX02}).
			%
			%
			Each solution $X$ of~\eqref{riccatic}  can also be associated with an \emph{extended Lagrangian invariant subspace}
			for the pencil $S_{c}(s)$ (see \cite{BenLMV15}), spanned by the columns of
			$ \widehat{U}_c:=\matr{ccc} -X^{\mathsf{T}}
& I_n & -F_c^{\mathsf{T}}  \rix^{\mathsf{T}}$.
In particular, $\widehat{U}_c$ satisfies
\[
	\left[ \begin{array}{ccc} 0 & A & B \\
	A^{\mathsf{H}} & 0 & C^{\mathsf{H}} \\  B^{\mathsf{H}} & C & R  \end{array} \right] \widehat{U}_c
	=\left[ \begin{array}{ccc} 0 & I_n & 0\\
	-I_n & 0 & 0\\ 0 & 0 & 0 \end{array} \right] \widehat{U}_c A_{F_c}.
\]
%
%
%
The sets $\XWpdc, \XWpdpdc$ are related to the concepts of \emph{passivity and strict passivity} see  \cite{Wil71}.
If for the system $\M:=\left\{A,B,C,D\right\}$ of (\ref{statespace}) the LMI (\ref{KYP-LMI})
has a solution $X\in\XWpdc$  then $\M$ is  \emph{(Lyapunov) stable} (i.e. all eigenvalues are in the closed left half plane with any eigenvalues occurring on the imaginary axis being semisimple), and \emph{passive}, and if there exists a solution $ X\in \XWpdpdc$ then $\M$ is \emph{asymptocially stable}, (i.e. all eigenvalues are the open left half plane) and \emph{strictly passive.} Furthermore, if $\M$ is passive, then
there exist maximal and minimal solutions $X_- \preceq X_+$ of \eqref{KYP-LMI} in $\XWpdc$
such that all solutions $X$ of $W_c(X)\succeq 0$ satisfy
\[
	0 \prec  X_- \preceq X \preceq X_+,
\]
which implies that $\XWpdc$ is bounded.
For more details on the different concepts discussed in this section, see \cite{BeaMV19}.

\subsection{Positive-realness and passivity, discrete-time}
For each of the results of the previous subsection there are discrete-time versions which we briefly recall in this section, see \cite{IonOW99,Pop73}.
Note, that these results can be obtained by applying a bilinear transform (see Appendix \ref{appendix:cayley}) to the continuous-time counterparts.

{
The transfer function $\mathcal T_d(s)$ in \eqref{ABCD} is called {\em positive real} \/if the 
		matrix-valued rational function
		\[ 
			\Phi_d(z):= \mathcal T_d^{\mathsf{H}}(z^{-1}) + \mathcal T_d(z)
		\] 
satisfies	$ \Phi_d(e^{\imath \omega}) = \Phi_d^{\mathsf{H}}(e^{\imath \omega}) \succeq 0 $ for $0 \leq \omega \leq 2\pi$,	 and it is called \emph{strictly positive real} if $\Phi_d(e^{\imath \omega}) \succ 0$ for  $0 \leq \omega \leq 2\pi$.

We consider an associated the matrix function
\begin{equation} \label{prlz}
	W_d(X) = \left[
		\begin{array}{cc}
			X-A^{\mathsf{H}}X\,A & C^{\mathsf{H}}-A^{\mathsf{H}}X\,B \\
			C-B^{\mathsf{H}}X\,A & B^{\mathsf{H}}X\,B +R
		\end{array}
	\right],
\end{equation}
where again $R=D+D^{\mathsf{H}}$,  the sets
\begin{subequations}\label{LMIsolnsetsd}
	\begin{align}
&\XWpdd :=\left\{ X\in \Hn \left|   W_d(X) \succeq 0,\ X \succ 0 \right.\right\}, 
\label{XpdsolnWpsdd} \\[1mm]
&\XWpdpdd :=\left\{ X\in \Hn \left|   W_d(X) \succ  0,\ X \succ 0 \right.\right\}. \label{XpdsolWpdd}
\end{align}
	\end{subequations}
	%
	%
	%
	%
	%
and the system pencil
	\begin{equation} \label{statespaced}
		S_{d}(z) =
		\left[ \begin{array}{cc|c} 0 & A-zI_n & B \\
		zA^{\mathsf{H}}-I_n & 0 & C^{\mathsf{H}} \\ \hline zB^{\mathsf{H}} & C &  R\end{array} \right]
	\end{equation}
	whose Schur complement is $\Phi_d(z)$.
	%
	%

	If the system is positive real}
	%
	%
	then, see \cite{Wil71}, there exists $X\in \Hn$ such that $W_d(X) \succeq 0$.
	%
	If that is the case, a transfer function ${\mathcal T}_d(z):=C(zI_n-A)^{-1}B+D$ is 
	called \emph{passive} and \emph{strictly passive} if even $W_d(X)\succ 0$.
	We again have an associated discrete-time Riccati equation   defined as
	\begin{equation}
		\mathsf{Ricc}_d(X) :=-A^{\mathsf{H}}XA+X-(C^{\mathsf{H}}-A^{\mathsf{H}} X B) (R-B^{\mathsf{H}} X B)^{-1}(C-B^{\mathsf{H}}X A) =0.\label{riccatid}
	\end{equation}
	from which one directly obtains a spectral factorization of $\Phi_d(z)$.
	The solutions of the discrete-time Riccati equation can be obtained by computing a Lagrangian invariant subspace spanned by the columns of $U_d:=\matr{cc} I_n & -X^{\mathsf{T}} \rix^{\mathsf{T}} $ of the \emph{symplectic matrix} %
	\[
		{\mathcal S}_d :=
		\left[\begin{array}{cc} I &  BR^{-1}B^{\mathsf{H}} \\
		0 & A^{\mathsf{H}}-C^{\mathsf{H}}R^{-1}B^{\mathsf{H}} \end{array} \right]^{-1}
		\left[\begin{array}{cc} A-BR^{-1}C & 0 \\
		C^{\mathsf{H}}R^{-1}C & I  \end{array} \right],
	\]
	%
	%
	%
	satisfying ${\mathcal S}_d U_d=U_d A_{F_d}$, where $ A_{F_d} :=  A-BF_d$ with $F_d := (R-B^{\mathsf{H}}XB)^{-1}(C-B^{\mathsf{H}}XA)$.

	Each solution $X$ of~\eqref{riccatid}  can also be associated with an \emph{extended Lagrangian invariant subspace}
	for the pencil $S_{d}(z)$ (see \cite{BenLMV15}), spanned by the columns of
	$ \widehat{U}_d:=\matr{ccc} -X^{\mathsf{T}}
& I_n & -F_d^{\mathsf{T}}  \rix^{\mathsf{T}}$.
In particular, $\widehat{U}_d$ satisfies
\[
	\left[ \begin{array}{ccc}
			0 & A & B \\
			I_n  & 0 & C^{\mathsf{H}} \\
			0 & C & R
	\end{array} \right] \widehat{U}_d
	=\left[ \begin{array}{ccc}
			0 & I_n & 0\\
			A^{\mathsf{H}} & 0 & 0\\
			B^{\mathsf{H}}  & 0 & 0
	\end{array} \right] \widehat{U}_d A_{F_d}.
\]
Again, if the system is passive, then there exist maximal and minimal solutions $X_- \preceq X_+$  in $\XWpdpdd$,
such that all solutions $X$ of $W_d(X)\succeq 0$ satisfy
\[
	0 \prec  X_- \preceq X \preceq X_+,
\]
which implies that $\XWpdd$ is bounded.

\section{The analytic center}\label{sec:ana_center}
If the sets $\XWpdpdc$, $\XWpdpdd$ in \eqref{LMIsolnsets}, respectively \eqref{LMIsolnsetsd}, are non-empty, then we can define their respective analytic center. Following the discussion in \cite{GenNV99}, we first consider the continuous-time case, the discrete-time case is derived in an analogous way. We choose a scalar barrier function
\begin{equation}
	b(X) := -\ln \det W_c(X),
\end{equation}
which is bounded from below but becomes infinitely large when $W_c(X)$ becomes singular.
We define the analytic center of the domain $\XWpdpdc$ as
the minimizer of this barrier function.

\subsection{The continuous-time case}\label{sec:dis-time}
The solutions $X_+$ and $X_-$ of the Riccati equation $\mathsf{Ricc}_{c}(X)=0$ in {\eqref{riccatic}}, are both on the boundary of $\XWpdc$, and hence are not in $\XWpdpdc$.
Since we assume that $\XWpdpdc$ is non-empty, the analytic center is well
defined, see, \eg, Section~4.2 in \cite{Nes13}.

To characterize the analytic center, we first need to find a variation of the \emph{gradient} $b_X$ of the barrier function $b$ at point $X$ along direction $\Delta_X$, which is equal to
\begin{equation}
	\sp{W_c(X)^{-1}}{\Delta W_c(X)[\Delta_X]},
\end{equation}
where $b_X = W_c(X)^{-1}$ and $\Delta W_c(X)[\Delta_X]$ is the incremental step in the direction $\Delta_X$, for details see Appendix~\ref{der_complex_functions}. It
appears that $X$ is an extremal point of the barrier function if and only if
\begin{equation}
	\sp{W_c(X)^{-1}}{\Delta W_c(X)[\Delta_X]} = 0\quad \text{for all }\Delta_X=\Delta_X^{\mathsf{H}}.
\end{equation}
The increment of $W_c(X)$ corresponding to an incremental direction $\Delta_X=\Delta_X^{\mathsf{H}}$ of $X$ is given by
\[
	\Delta W_c(X)[\Delta_X] = -\left[ \begin{array}{cc}
	A^{\mathsf{H}}\Delta_X+\Delta_X A & \Delta_X B \\ B^{\mathsf{H}}\Delta_X & 0 \end{array} \right].
\]
The equation for the extremal point then becomes
\begin{equation}\label{orth}
	\sp{W_c(X)^{-1}}{\left[ \begin{array}{cc}
	A^{\mathsf{H}}\Delta_X + \Delta_X A & \Delta_X B \\ B^{\mathsf{H}}\Delta_X & 0 \end{array} \right]}
	=0\quad \text{for all }\Delta_X=\Delta_X^{\mathsf{H}}.
\end{equation}
Defining
\[
	F_c := R^{-1}(C-B^{\mathsf{H}}X), \quad P_c := -A^{\mathsf{H}}X-XA-F_c^{\mathsf{H}}RF_c,
\]
then
\[ 
	W_c(X)= \left[ \begin{array}{cc}
	I & F_c^{\mathsf{H}} \\ 0 & I \end{array} \right]
	\left[ \begin{array}{cc}
	P_c & 0 \\ 0 & R \end{array} \right]
	\left[ \begin{array}{cc}
			I & 0 \\ F_c
			  & I \end{array} \right].
		\]
		For a point $X\in\XWpdpdc$ it is obvious that
		we also have $P_c=\mathsf{Ricc}_c(X)\succ  0$, and hence (\ref{orth}) is equivalent to
		\[ 
			\sp{\left[ \begin{array}{cc}
				P_c^{-1} & 0 \\ 0 & R^{-1} \end{array} \right]
				}{\left[ \begin{array}{cc}
				I & -F_c^{\mathsf{H}} \\ 0 & I \end{array} \right]
				\left[ \begin{array}{cc}
				A^{\mathsf{H}}\Delta_X+\Delta_X A & \Delta_X B \\ B^{\mathsf{H}}\Delta_X & 0 \end{array} \right]
				\left[ \begin{array}{cc}
				I & 0 \\ -F_c & I \end{array} \right]
			} = 0,
			\]
			or 
			\[ 
				\sp{ P_c^{-1} }{ A^{\mathsf{H}}\Delta_X+\Delta_X A-F_c^{\mathsf{H}}B^{\mathsf{H}}\Delta_X-\Delta_X BF_c } =0,
			\] 
			and this is equivalent to
			\begin{equation} \label{skew}
				P_c^{-1}A_{F_c}^{\mathsf{H}} +A_{F_c} P_c^{-1} =0,
			\end{equation}
			where we have set $A_{F_c} = A-BF_c$.

			We emphasize that $P_c$ is nothing but the Riccati operator
			$\mathsf{Ricc}_c(X)$ defined in \eqref{riccatic}, and that $A_{F_c}$ is the corresponding
			closed loop matrix. For the classical Riccati solutions we have
			$P_c=\mathsf{Ricc}_c(X)=0$ and the corresponding closed loop matrix
			is well-known to have its eigenvalues equal to a subset of the
			eigenvalues of the corresponding Hamiltonian matrix \eqref{HamMatrix}.

			Since $P_c=\mathsf{Ricc}_c(X)\succ  0$, it follows that $P_c$ has a
			Hermitian square root $T_c$ satisfying $P_c=T_c^2$. Transforming
			(\ref{skew}) with the invertible matrix $T_c$,
			we obtain
			\[
				T_c^{-1}A_{F_c}^{\mathsf{H}}T_c + T_cA_{F_c}T_c^{-1}=0.
			\]
			Hence $\hat A_{F_c}:=T_cA_{F_c}T_c^{-1}$ is skew-Hermitian and has all its eigenvalues on
			the imaginary axis, and so does $A_{F_c}$.
			Therefore, the closed loop matrix $A_{F_c}$ of the analytic center has
			a spectrum that is also central.

			It is important to also note that
			\[
				\det W_c(X) = \det \mathsf{Ricc}_c(X)  \det R,
			\]
			which implies that we are also finding a stationary point of $\det \mathsf{Ricc}_c(X)$, since $\det R$ is constant and non-zero.

			Since the matrix $P_c$ is positive definite and invertible, we can rewrite
			the equations defining the analytic center as
			\begin{subequations}
				\begin{align} \label{Fc}
					RF_c  &= C-B^{\mathsf{H}}X,\\ \label{Xc}
					P_c   &= -A^{\mathsf{H}}X-X A-F_c^{\mathsf{H}}RF_c, \\ \label{Pc}
					0 &= P_c(A-BF_c)+(A^{\mathsf{H}}-F_c^{\mathsf{H}}B^{\mathsf{H}})P_c,
				\end{align}
			\end{subequations}
			where $X=X^{\mathsf{H}}$ and $P_c=P_c^{\mathsf{H}}\succ  0$.
			We can compute the analytic center by solving these three equations which actually form a cubic equation in $X$.

			Note that even though the eigenvalues of the closed loop matrix $F_c$ associated with the analytic center are all purely imaginary, the eigenvalues of the original system and the poles of the transfer function stay invariant under the state space transformation $\mathcal T_c$.

			\subsection{The discrete-time case}\label{sec:discase}
			For discrete-time systems, the increment of $W_d(X)$ equals
			\begin{equation}
				\Delta W_d(X)[\Delta_X] =  - \left[ \begin{array}{cc}
				A^{\mathsf{H}}\Delta_XA-\Delta_X & A^{\mathsf{H}}\Delta_XB \\ B^{\mathsf{H}}\Delta_XA & B^{\mathsf{H}}\Delta_XB  \end{array} \right],
			\end{equation}
			for all $\Delta_X=\Delta_X^{\mathsf{H}}$.
			Defining $F_d := (R-B^{\mathsf{H}}XB)^{-1}(C - B^{\mathsf{H}}XA)$, $P_d := - A^{\mathsf{H}}XA + X-F_d^{\mathsf{H}}(R-B^{\mathsf{H}}XB)F_d$, and $A_{F_d} := A-BF_d$, then $W_d(X)$ factorizes as
			\begin{equation}
				W_d(X)= \left[ \begin{array}{cc}
				I & F_d^{\mathsf{H}} \\ 0 & I \end{array} \right]
				\left[ \begin{array}{cc}
				P_d & 0 \\ 0 & R-B^{\mathsf{H}}XB \end{array} \right]
				\left[ \begin{array}{cc}
						I & 0 \\ F_d
						  & I \end{array} \right],
					\end{equation}
					and the equation for the extremal point becomes
					\begin{align}
& \Bigg\langle \left[ \begin{array}{cc}
	P_d^{-1} & 0 \\ 0 & (R-B^{\mathsf{H}}XB)^{-1} \end{array} \right]
	, \\
		 & \qquad \left[ \begin{array}{cc}
		 I & -F_d^{\mathsf{H}} \\ 0 & I \end{array} \right]
		 \left[ \begin{array}{cc}
		 A^{\mathsf{H}}\Delta_X A-\Delta_X & A^{\mathsf{H}}\Delta_X B \\ B^{\mathsf{H}}\Delta_X A & B^{\mathsf{H}}\Delta_X B \end{array} \right]
		 \left[ \begin{array}{cc}
		 I & 0 \\ -F_d & I \end{array} \right]
	 \Bigg \rangle =0,
	 \end{align}
	 or
	 \begin{equation}
		 \sp{P_d^{-1} }{ A_{F_d}^{\mathsf{H}}\Delta_X A_{F_d}-\Delta_X } + \sp{(R-B^{\mathsf{H}}XB)^{-1}}{B^{\mathsf{H}}\Delta_X B} =0.
	 \end{equation}
	 This is equivalent to
	 \begin{equation} \label{stein}
		 A_{F_d}P_d^{-1}A_{F_d}^{\mathsf{H}} -P_d^{-1}+B(R-B^{\mathsf{H}}XB)^{-1}B^{\mathsf{H}} =0,
	 \end{equation}
	 which is a non-homogenous discrete-time Lyapunov equation. Since $(A,B)$
	 is controllable (by assumption), so is $(A_{F_c},B)$ and it follows
	 then from (\ref{stein}) that the eigenvalues of $A_{F_d}$ are now
	 strictly inside the unit circle. This is \emph{clearly different from the
	 continuous-time case}, where the spectrum of $A_{F_c}$ was on the boundary
	 of the stability region. The equations defining the discrete-time analytic center then become
	 \begin{subequations}
		 \begin{align} \label{Fd}
			 (R-B^{\mathsf{H}}X B)F_d  &= C-B^{\mathsf{H}}X A,\\ \label{Xd}
			 P_d   &= C^{\mathsf{H}}R^{-1}C+X-A^{\mathsf{H}}X A\nonumber \\
			       & \qquad-F_d^{\mathsf{H}}(R-B^{\mathsf{H}}X B)F_d, \\ \nonumber
			 0&= (A-BF_d)P_d^{-1}(A^{\mathsf{H}}-F_d^{\mathsf{H}}B^{\mathsf{H}})\\
			  &  \qquad -P_d^{-1}+
			  B(R-B^{\mathsf{H}}X B)^{-1}B^{\mathsf{H}}.\label{Pd}
		 \end{align}
	 \end{subequations}

	 \begin{remark}\label{rem:inside}{\rm
			 Note that we could have transformed the solution of the corresponding
			 continuous-time problem via a bilinear transform, which would then yield
			 a feedback $F_d$ that puts all eigenvalues on the unit circle, but the feedback
			 would of course be different. For a more detailed discussion, see
			 Appendix~\ref{appendix:cayley}.
		 }
	 \end{remark}
	 \section{Numerical computation of the analytic center}\label{sec:numerical}
	 In this section we present methods for the numerical computation of the analytic center.

	 Suppose that we are at a point $X_0\in \XWpdpdc\ (\XWpdpdd)$ and want to perform the next step using an increment $\Delta_X$.
	 We discuss a steepest ascent and a Newton-like method to obtain that increment.

	 \subsection{A steepest ascent method}\label{sec:ascent}
	 In order to formulate an optimization scheme to compute the analytic center, we can use the gradient of the barrier function $b(X)$ with respect to $X$ to obtain a steepest ascent method.

	 In the continuous-time case, we then need to take a step $\Delta_X$ for which
	 $\sp{b_W(X_{0})}{  \Delta W_c(X_{0})[\Delta_X]}$ is maximal,
	 which is equivalent to
	 \[
		 \Delta_X := \argmax_{\sp{\Delta_X}{\Delta_X}=1} \sp{
		 P_c^{-1}(X_{0})A_{F_c}(X_{0})^{\mathsf{H}}+A_{F_c}(X_{0})P_c^{-1}(X_{0})}{ \Delta_X }.
	 \]
	 The maximum is obtained by choosing $\Delta_X$ proportional to the gradient %
	 \[
		 P_c^{-1}(X_{0})A_{F_c}(X_{0})^{\mathsf{H}}+A_{F_c}(X_{0})P_c^{-1}(X_{0}).
	 \]
	 The corresponding optimal stepsize $\alpha$ for the increment  $\Delta_X$ can be obtained from the determinant of the incremented LMI $W_c(X_{0} + \alpha \Delta_{X})\succ 0$. 

	 In  the discrete-time case, we obtain the increment from
		 \begin{align}& \Delta_X := \\  &\quad \argmax_{\sp{\Delta_X}{\Delta_X}=1}
			 \sp{ A_{F_d}(X_{0}) P_d^{-1}(X_{0})A_{F_d}^{\mathsf{H}}(X_{0})-P_d^{-1}(X_{0})  +  B(R-B^{\mathsf{H}}X_{0}B)^{-1}B^{\mathsf{H}}}{ \Delta_X }.
		 \end{align}
		 The maximum is obtained by choosing $\Delta_X$ proportional to
		 \[
			 A_{F_d}(X_{0}) P_d^{-1}(X_{0})A_{F_d}^{\mathsf{H}}(X_{0})-P_d^{-1}(X_{0})  + B(R-B^{\mathsf{H}}X_{0}B)^{-1}B^{\mathsf{H}},
		 \]
		 and the stepsize $\alpha$ for the increment $\Delta_X$ can again be obtained from the determinant of the incremented LMI $W_{d}(X_{0} + \alpha \Delta_{X})\succ 0$. 

		 \begin{remark}\label{remstep}{\rm
				 The detailed explanation how to compute the stepsize $\alpha$ will be done later as a special case of the derivation of the Newton step, see subsection~\ref{sec:newton}.
				 The idea is to find the second order Taylor expansion of the function
				 $f(X_{0}+\alpha \Delta_X)=-\ln \det G(X_{0}+\alpha \Delta_X)$ and then to maximize this
				 quadratic function in the scalar $\alpha$.
			 }
		 \end{remark}

		 \subsection{A Newton-like method}\label{sec:newton}
		 For the computation of a Newton-like increment $\Delta_X$ we also need the Hessian of the barrier function~$b$. In order to simplify the derivation we first equivalently reformulate the barrier function into a more suitable form.

		 \subsubsection{The continuous-time case}

		 In   the continuous-time case, we have that
		 \[
			 W_c (X_0+\Delta_X) = \left[
				 \begin{array}{cc}
					 Q_0 & C_0^{\mathsf{H}}\\
					 C_0 & R_0
				 \end{array}
			 \right]
			 -  \left[\begin{array}{c}
					 \Delta_X\\ 0
				 \end{array}
				 \right]  \left[\begin{array}{cc}
					 A & B
				 \end{array}
			 \right] -
			 \left[\begin{array}{c}
					 A^{\mathsf{H}} \\ B^{\mathsf{H}}
				 \end{array}
				 \right]  \left[\begin{array}{cc}
					 \Delta_X & 0
				 \end{array}
			 \right],
		 \]
		 where
		 \[\left[
				 \begin{array}{cc}
					 Q_0 & C_0^{\mathsf{H}}\\
					 C_0 & R_0
				 \end{array}
			 \right] := W_c (X_0).
		 \]
		 Up to the constant $(-1)^n$, the determinant of $W_c (X_0+\Delta_X) $  is equal to
		 \begin{equation}
			 \label{transfo}
			 \det \left[\begin{array}{cc|cc}
					 0 & I_n & \Delta_X & 0 \\
					 I_n & 0 & A & B \\ \hline
					 \Delta_X & A^{\mathsf{H}} & Q_0 & C_0^{\mathsf{H}} \\
					 0 & B^{\mathsf{H}} &C_0 & R_0
			 \end{array}\right]
			 = \det \left[\begin{array}{cc|cc}
					 0 & I_n & \Delta_X & 0 \\
					 I_n & 0 & A_{F_c} & B \\ \hline
					 \Delta_X & A_{F_c}^{\mathsf{H}} & P_0 & 0 \\
					 0 & B^{\mathsf{H}} & 0 & R_0
			 \end{array}\right],
		 \end{equation}
		 where $A_{F_c}:=A-BR_0^{-1}C_0$ and $P_0:=Q_0-C_0^{\mathsf{H}}R_0^{-1}C_0$ are associated with the current point $X_0$. Carrying out an additional congruence transformation with
		 \[
			 Z_c := \left[\begin{array}{cccc} P_0^{-\frac12} & 0 & 0 & 0  \\ 0 & P_0^\frac12 & 0 & -\hat BR_0^{-1} \\ 0 & 0 & P_0^{-\frac12}& 0 \\ 0 & 0 & 0 & I_m \end{array}\right],
		 \]
		 we obtain
		 \[
			 \left[\begin{array}{cccc}
					 0 & I_n & \hat \Delta_X & 0 \\
					 I_n & -\hat BR_0^{-1}\hat B^{\mathsf{H}} & \hat A_{F_c} & 0 \\
					 \hat \Delta_X & \hat A_{F_c}^{\mathsf{H}} & I_n & 0 \\
					 0 & 0 & 0 & R_0
			 \end{array}\right] :=
			 Z_c \left[\begin{array}{cccc}
					 0 & I_n & \Delta_X & 0 \\
					 I_n & 0 & A_{F_c} & B \\
					 \Delta_X & A_{F_c}^{\mathsf{H}} & P_0 & 0 \\
					 0 & B^{\mathsf{H}} & 0 & R_0
			 \end{array}\right] Z_c^{\mathsf{H}},
		 \]
		 where $\hat B = P_0^\frac12 B$, $\hat A_{F_c}:=P_0^\frac12A_FP_0^{-\frac12}$, and $\hat \Delta_X=P_0^{-\frac12}\Delta_XP_0^{-\frac12}$.
		 It is clear that the determinant of the congruence transformation introduces a factor $\det (P_0)$. Finally, the determinant of the transformed matrix is, up to a constant $\det (R_0)$, equal to
		 \begin{multline}  \det \left[\begin{array}{cc}
				 -\hat \Delta_X & I_n \end{array}\right] \left[\begin{array}{cc}
				 -\hat BR_0^{-1}\hat B^{\mathsf{H}} & \hat A_{F_c} \\
				 \hat A_{F_c}^{\mathsf{H}} & I_n
				 \end{array}\right]  \left[\begin{array}{cc}
			 -\hat \Delta_X \\  I_n \end{array}\right]\\  = \det \left[I_n  - \hat \Delta_X \hat A_{F_c} -   \hat A_{F_c}^{\mathsf{H}} \hat \Delta_X   -  \hat \Delta_X\hat BR_0^{-1}\hat B^{\mathsf{H}} \hat \Delta_X
		 \right].
			 \end{multline}
			 This is the multiplying factor of the current value of $\det W_c(X_0)$ and we can make it larger than $1$ if $\hat A_{F_c}$ is not skew-Hermitian yet. Introduce
			 \begin{align}
				 f(X)&:= - \ln \det ( G(X)),\\
				 Q_c&:=\hat BR_0^{-1}\hat B^{\mathsf{H}},\\
				 G(X)&:= I_n  - X \hat A_{F_c} -   \hat A_{F_c}^{\mathsf{H}} X   -  XQ_cX.
			 \end{align}
			 In the set of Hermitian matrices (over the reals), the gradient of $f(X)$ then  is given by
			 \[
				 f_X(X)[\Delta] = \langle- G(X)^{-1},-(  \Delta \hat A_{F_c} +   \hat A_{F_c}^{\mathsf{H}} \Delta   +  \Delta Q_c X +  X Q_c \Delta )\rangle
			 \]
			 and the Hessian is given by
			 \begin{align}
			 f_{XX}(X)[\Delta,\Delta] &= \left < -G(X)^{-1}(\Delta \hat A_{F_c} +   \hat A_{F_c}^{\mathsf{H}} \Delta   +  \Delta  Q_c X +  X Q_c\Delta)G(X)^{-1},\right .\\
						  & \qquad \left . -(\Delta \hat A_{F_c} +   \hat A_{F_c}^{\mathsf{H}} \Delta   +  \Delta  Q_c X +  X Q_c\Delta)\right >  \\ &\qquad+ \langle -G(X)^{-1}, -2\Delta Q_c \Delta  \rangle.
		 \end{align}
		 A second order approximation of $f$ (at $X=0$) is given by
		 \begin{align}
			 f(\Delta) \approx T^{(2)}_f(\Delta) &= f(0) +f_X(0)[\Delta] + \frac12f_{XX}(0)[\Delta, \Delta]\\
							     & = \langle I_n, \Delta \hat A_{F_c} +   \hat A_{F_c}^{\mathsf{H}} \Delta   \rangle + \frac12
							     \langle \Delta \hat A_{F_c} +   \hat A_{F_c}^{\mathsf{H}} \Delta  , \Delta \hat A_{F_c} +   \hat A_{F_c}^{\mathsf{H}} \Delta \rangle \\
							     &\qquad +  \langle I_n, \Delta  Q_c \Delta \rangle,
		 \end{align}
		 and we want the gradient of $f$ to be $0$. For the Newton step we want to determine $\Delta=\Delta^{\mathsf{H}}$ such that $\frac{\partial T^{(2)}_f}{\partial \Delta}(\Delta)=0$,
		 \ie we require that
		 \begin{equation}
			 \langle I_n,  Y \hat A_{F_c}  + \hat A_{F_c}^{\mathsf{H}} Y\rangle
			 +\langle  \Delta \hat A_{F_c} +   \hat A_{F_c}^{\mathsf{H}} \Delta, Y \hat A_{F_c} +   \hat A_{F_c}^{\mathsf{H}} Y\rangle
			 + 2\langle I_n, Y Q_c \Delta\rangle = 0
		 \end{equation}
		 for all $Y=Y^{\mathsf{H}}$.
		 Using the properties of the scalar product, we obtain that this is equivalent to
		 \[
			 \langle Y,
			 \hat A_{F_c}^{\mathsf{H}} + \hat A_{F_c} +\hat A_{F_c}\Delta \hat A_{F_c}  + \hat A_{F_c} \hat A_{F_c}^{\mathsf{H}} \Delta + \hat A_{F_c}^{\mathsf{H}}\Delta \hat A_{F_c}^{\mathsf{H}} +  \Delta\hat A_{F_c} \hat A_{F_c}^{\mathsf{H}} + Q_c \Delta + \Delta Q_c \rangle =0
		 \]
		 for all $Y= Y^{\mathsf{H}}$, or equivalently
		 \[
			 \hat A_{F_c}\Delta \hat A_{F_c}  + \hat A_{F_c} \hat A_{F_c}^{\mathsf{H}} \Delta + \hat A_{F_c}^{\mathsf{H}}\Delta \hat A_{F_c}^{\mathsf{H}} +  \Delta\hat A_{F_c} \hat A_{F_c}^{\mathsf{H}} + Q_c \Delta + \Delta Q_c =  -\hat A_{F_c}^{\mathsf{H}}-\hat A_{F_c}.
		 \]
		 If we fix a direction $\Delta$ and look for $\alpha$ such that $f(\alpha\Delta)$ is maximal, then the Newton step can be computed in an analogous way. With $g(\alpha)= f(\alpha\Delta)$, we then have
		 \[
			 g(\alpha)\approx f(0) + \alpha f_{X}(0)[\Delta] + \frac12 \alpha^2f_{XX}(0)[\Delta, \Delta]
		 \]
		 and thus the Newton correction in $\alpha$ is given by
		 \begin{equation}\label{eq:newtonct}
			 \delta_\alpha = - \frac{\langle I_n, \Delta \hat A_{F_c} + \hat A_{F_c}^{\mathsf{H}} \Delta \rangle}{\langle I_n, \Delta Q_c \Delta \rangle + \frac12 \|\Delta \hat A_{F_c} + \hat A_{F_c}^{\mathsf{H}} \Delta \|^2}.
		 \end{equation}
		 \subsubsection{The discrete-time case}
		 For  the discrete-time case, we have that
		 \[
			 W_d (X_0+\Delta_X) = \left[
				 \begin{array}{cc}
					 Q_0 & C_0^{\mathsf{H}}\\
					 C_0 & R_0
				 \end{array}
			 \right]
			 -  \left[\begin{array}{c}
					 A^{\mathsf{H}} \\ B^{\mathsf{H}}
				 \end{array}
				 \right]  \Delta_X \left[\begin{array}{cc}
					 A & B
				 \end{array}
			 \right] +
			 \left[\begin{array}{c}
					 I_n \\ 0
				 \end{array}
				 \right]  \Delta_X \left[\begin{array}{cc}
					 I_n & 0
				 \end{array}
			 \right],
		 \]
		 where
		 \[ \left[
				 \begin{array}{cc}
					 Q_0 & C_0^{\mathsf{H}}\\
					 C_0 & R_0
				 \end{array}
			 \right] := W_d (X_0).
		 \]
		 The determinant of $W_d (X_0+\Delta_X)$ is, up to the constant $(-1)^n$, equal to
		 \begin{equation}
			 \label{transfod}
			 \det \left[\begin{array}{cc|cc}
					 -I_n & 0 & \Delta_X & 0 \\
					 0 & I_n & A & B \\ \hline
					 I_n & A^{\mathsf{H}} \Delta_X & Q_0 & C_0^{\mathsf{H}} \\
					 0 & B^{\mathsf{H}} \Delta_X  & C_0 & R_0
			 \end{array}\right]
			 = \det \left[\begin{array}{cc|cc}
					 -I_n & 0 & \Delta_X & 0 \\
					 0 & I_n & A_{F_d} & B \\ \hline
					 I_n & A_{F_d}^{\mathsf{H}}\Delta_X & P_0 & 0 \\
					 0 & B^{\mathsf{H}}\Delta_X & 0 & R_0
			 \end{array}\right]
		 \end{equation}
		 where $R_0=R-B^{\mathsf{H}}X_0B$, $C_0=C-B^{\mathsf{H}}X_0A$, $Q_0=-A^{\mathsf{H}}X_0A+X_0$, $A_{F_d}:=A-BR_0^{-1}C_0^{\mathsf{H}}$ and $P_0:=Q_0-C_0R_0^{-1}C_0^{\mathsf{H}}$ are associated with the current point $X_0$. Setting
		 \[
			 Z_\ell := \left[\begin{array}{cccc} P_0^{-\frac12} & 0 & 0 & 0  \\ 0 &  P_0^{\frac12} & 0 & -\hat BR_0^{-1} \\ 0 & 0 & P_0^{-\frac12}& 0 \\ 0 & 0 & 0 & I_m \end{array}\right], \   Z_r := \left[\begin{array}{cccc} P_0^{\frac12} & 0 & 0 & 0  \\ 0 &  P_0^{-\frac12} & 0 & 0 \\ 0 & 0 & P_0^{-\frac12}& 0 \\ 0 & -R_0^{-1}\hat B^{\mathsf{H}}\hat \Delta_X & 0 & I_m \end{array}\right],
		 \]
		 transforming with $Z_\ell$ from the left and $Z_r$ from the right, and substituting $\hat B = P_0^\frac12 B$, $\hat A_{F_d}:=P_0^\frac12A_{F_d}P_0^{-\frac12}$, and $\hat \Delta_X=P_0^{-\frac12}\Delta_XP_0^{-\frac12}$, we obtain the matrix
		 \[
			 \left[\begin{array}{cccc}
					 -I_n & 0 & \hat \Delta_X & 0 \\
					 0 & I_n -\hat BR_0^{-1}\hat B^{\mathsf{H}}\hat \Delta_X & \hat A_{F_d} & 0 \\
					 I_n & \hat A_{F_d}^{\mathsf{H}}\hat \Delta_X & I_n & 0 \\
					 0 & 0 & 0 & R_0
			 \end{array}\right] :=
			 Z_\ell \left[\begin{array}{cccc}
					 -I_n & 0 & \Delta_X & 0 \\
					 0 & I_n & A_{F_d} & B \\
					 I_n & A_{F_d}^{\mathsf{H}}\Delta_X  & P_0 & 0 \\
					 0 & B^{\mathsf{H}}\Delta_X  & 0 & R_0
			 \end{array}\right] Z_r.
		 \]
		 The determinant of the transformed matrix is equal to
		 \[
			 \det \left[\begin{array}{cc}
					 I_n -\hat BR_0^{-1}\hat B^{\mathsf{H}}\hat \Delta_X & \hat A_{F_d} \\
					 \hat A_{F_d}^{\mathsf{H}}\hat \Delta_X & I_n +\hat \Delta_X \\
			 \end{array}\right] \cdot \det R_0.
		 \]
		 We introduce
		 \begin{align}
			 f(X)&:= - \ln \Re \det ( G(X)), \\ G(X)&:= \left[\begin{array}{cc}
					 I_n -Q_d X & \hat A_{F_d} \\
					 \hat A_{F_d}^{\mathsf{H}} X & I_n +X \\
			 \end{array}\right], \\ Q_d&:=\hat BR_0^{-1}\hat B^{\mathsf{H}},
				 \end{align}
				 and compute the gradient and the Hessian of $f(X)$. The computation of the gradient is not as straight-forward  as in the continuous-time case, since we consider non-Hermitian matrices. It is given by
				 \begin{equation}
					 f_{X}(X)[\Delta] =\sp{-\frac{\overline{\det G(X)}}{\Re\det G(X)}G(X)^{-\mathsf{H}}}{ \left[\begin{array}{cc}
								 -Q_d \Delta & 0 \\
								 \hat A_{F_d}^{\mathsf{H}}\Delta & \Delta \\
					 \end{array}\right]},
				 \end{equation}
				 see Appendix~\ref{der_complex_functions} for more details.
				 Revisiting the steps for the derivation of $G(X)$, we notice that $\det (G(X))$ is still real and the solution $\Delta$ is still unique and Hermitian.
				 Thus, the Hessian is given by
				 \begin{equation}
					 f_{XX}(X)[\Delta,\Delta] =\sp{G(X)^{-\mathsf{H}}
						 \left[\begin{array}{cc}
								 -Q_d \Delta & 0 \\
								 \hat A_{F_d}^{\mathsf{H}}\Delta & \Delta \\
						 \end{array}\right]^{\mathsf{H}}
						 G(X)^{-\mathsf{H}}}{\left[\begin{array}{cc}
								 -Q_d \Delta & 0 \\
								 \hat A_{F_d}^{\mathsf{H}}\Delta & \Delta \\
					 \end{array}\right]},
				 \end{equation}
				 and a second order approximation of $f$ (at $X=0$) is given by
				 \begin{align}
					 f(\Delta) &\approx T^{(2)}_f(\Delta)\\ &= f(0) +f_X(0)[\Delta] + \frac12f_{XX}(0)[\Delta, \Delta]\\
						   &=
						   -\sp{
							   \begin{bmatrix}
								   I_n & 0\\
								   -\hat A_{F_d}^{\mathsf{H}} 	& I_n
							   \end{bmatrix}
							   }{
							   \left[\begin{array}{cc}
									   -Q_d Y & 0 \\
									   \hat A_{F_d}^{\mathsf{H}}Y & Y \\
							   \end{array}\right]
						   }\\
						   &+\frac12
						   \sp{
							   \begin{bmatrix}
								   I_n & 0\\
								   -\hat A_{F_d}^{\mathsf{H}} 	& I_n
							   \end{bmatrix}
							   \left[\begin{array}{cc}
									   -\Delta Q_d & \Delta \hat A_{F_d} \\
									   0 & \Delta \\
							   \end{array}\right]
							   \begin{bmatrix}
								   I_n & 0\\
								   -\hat A_{F_d}^{\mathsf{H}} 	& I_n
							   \end{bmatrix}}{
							   \left[\begin{array}{cc}
									   -Q_d \Delta & 0 \\
									   \hat A_{F_d}^{\mathsf{H}}\Delta & \Delta \\
							   \end{array}\right]
						   } \\
						   &=
						   -\sp{ I_n - Q_d -\hat A_{F_d} \hat{A}_F^{\mathsf{H}}}{ \Delta}\\
						   & +\frac12
						   \sp{
							   Q_d\Delta Q_d\Delta - 2\hat A_{F_d} \hat A_{F_d}^{\mathsf{H}} \Delta Q_d \Delta - \hat A_{F_d} \hat A_{F_d}^{\mathsf{H}} \Delta \hat A_{F_d} \hat A_{F_d}^{\mathsf{H}} \Delta
						   + 2\hat A_{F_d}  \Delta \hat A_{F_d}^{\mathsf{H}}  \Delta - \Delta^2}
						   {I_n}.
				 \end{align}
				 We want the gradient of $f$ to be $0$, so for the Newton step we determine $\Delta=\Delta^{\mathsf{H}}$ such that $\frac{\partial T^{(2)}_f}{\partial \Delta}(\Delta)=0$,
				 or equivalently
				 \begin{align}
& 0=\langle I_n - Q_d -\hat A_{F_d} \hat A_{F_d}^{\mathsf{H}}, Y\rangle + \langle Q_d\Delta Q_d + \hat A_{F_d} \hat A_{F_d}^{\mathsf{H}} \Delta Q_d, Y  \rangle\\
& \qquad+
\langle
Q_d \Delta\hat A_{F_d} \hat A_{F_d}^{\mathsf{H}}  + \hat A_{F_d} \hat A_{F_d}^{\mathsf{H}} \Delta \hat A_{F_d} \hat A_{F_d}^{\mathsf{H}}   - \hat A_{F_d}  \Delta \hat A_{F_d}^{\mathsf{H}}- \hat A_{F_d}^{\mathsf{H}}  \Delta \hat A_{F_d}  + \Delta,
Y\rangle
				 \end{align}
				 for all $Y=Y^{\mathsf{H}}$.
				 Using the properties of the scalar product, we obtain that
				 \begin{multline}\label{eq:newtondiscrete}
					 I_n - Q_d -\hat A_{F_d} \hat A_{F_d}^{\mathsf{H}} \\  =
					 Q_d\Delta Q_d + \hat A_{F_d} \hat A_{F_d}^{\mathsf{H}} \Delta Q_d  +Q_d \Delta\hat A_{F_d} \hat A_{F_d}^{\mathsf{H}}  + \hat A_{F_d} \hat A_{F_d}^{\mathsf{H}} \Delta \hat A_{F_d} \hat A_{F_d}^{\mathsf{H}} \\
					 - \hat A_{F_d}  \Delta \hat A_{F_d}^{\mathsf{H}}  - \hat A_{F_d}^{\mathsf{H}}  \Delta \hat A_{F_d}  + \Delta.
				 \end{multline}

				 If we fix a direction $\Delta$ and look for $\alpha$ such that $f(\alpha\Delta)$ is maximal,
				 then the Newton correction in $\alpha$ is given by
				 \begin{equation}\label{eq:newtondt}
					 \delta_\alpha = \frac{2\sp{I_n - Q_d -\hat A_{F_d} \hat{A}_F^{\mathsf{H}}}{\Delta}}{
						 \sp{
							 Q_d\Delta Q_d\Delta - 2\hat A_{F_d} \hat A_{F_d}^{\mathsf{H}} \Delta Q_d \Delta - \hat A_{F_d} \hat A_{F_d}^{\mathsf{H}} \Delta \hat A_{F_d} \hat A_{F_d}^{\mathsf{H}} \Delta
						 + 2\hat A_{F_d}  \Delta \hat A_{F_d}^{\mathsf{H}}  \Delta - \Delta^2}
					 {I_n}}.
				 \end{equation}
				 \begin{remark}{\rm
						 To  carry out the Newton step,  we have to solve equation \eqref{eq:newtonct} in the continuous-time case or \eqref{eq:newtondiscrete} in the discrete-time case. This can be done via Kronecker products (for the cost of increasing the system dimension to $n^2$), \ie via
						 \begin{align}
& \left((I_n \otimes \hat{A}_{F_c} + \overline{\hat{A}_{F_c}} \otimes I_n)
	(\hat{A}_{F_c}^T
	\otimes I_n + I_n \otimes\hat{A}_{F_c}^{\mathsf{H}}) + I_n \otimes Q_c + \overline{Q}_c
\otimes I_n\right) \vect X\\
& \quad = \vect( \hat{A}_{F_c} + \hat{A}_{F_c}^{\mathsf{H}})
						 \end{align}
						 in the continuous-time case, or
						 \begin{align}
	&
	\left(
		(\overline{\hat{A}_{F_d}} \otimes \hat{A}_{F_d} - I_n \otimes I_n)({\hat{A}_{F_d}^T}
		\otimes \hat{A}_{F_d}^{\mathsf{H}} - I_n \otimes I_n) + \overline Q_d \otimes
	\hat{A}_{F_c}\hat{A}_{F_c}^{\mathsf{H}} \right . \\
	& \qquad \left .+ \overline{\hat{A}_{F_c}}\hat{A}_{F_c}^T \otimes Q_d
	+ \overline{Q_d}\otimes Q_d \right) \vect X	  = \vect(I_n -Q_d - \hat{A}_{F_d} \hat{A}_{F_d}^{\mathsf{H}})
						 \end{align}
						 in the discrete-time case.
					 }
				 \end{remark}
				 \subsubsection{Convergence}
				 In this subsection, we show that the functions that we consider here actually have a globally
				 converging Newton method. For this we have to analyze some more
				 properties of our functions and refer to \cite{BoyV2004,Nes13} for more details.
				 %
				 Recall that a smooth function $f: \R^n\rightarrow \R$ is
				 \emph{self-concordant} if it is a closed and convex function with open domain and
				 \begin{equation}
					 |f^{(3)}(x)| \le 2 f^{(2)}(x)^{\frac32}
				 \end{equation}
				 in the case $n=1$, and if $n>1$, then $f$ is \emph{self-concordant} if it is self-concordant along
				 every direction in its domain. In particular, if $n=1$ then $f(x) = - \ln(x)$ is self-concordant and in general, if $f$ is self-concordant and in addition $A\in\mat nm$,
				 $b\in\R^n$, then $f(Ax + b)$ is also self-concordant.
				 These results can be easily extended to the real space of complex matrices showing that the function $b(X)= - \ln\det(W(X))$ is self-concordant.
				 Let
				 \begin{equation}
					 \lambda(X):= \sp{\left(b_{XX}\right)^{-1}b_X}{b_X},
				 \end{equation}
				 where  $\left(b_{XX}\right)^{-1}b_X = \Delta$ in the Newton step, \ie {$\lambda(X) = \sp{\Delta}{A_{F_c}P_c^{-1} + P_c^{-1} A_{F_c}^{\mathsf{H}}}$ in the continuous-time case, or  $\lambda(X) = \sp{\Delta}{A_{F_d}P_d^{-1}A_{F_d}^{\mathsf{H}} -P_d^{-1} + B(R-B^{\mathsf{H}}X B)^{-1}B^{\mathsf{H}}}$ in the discrete-time case} respectively. In both cases $\lambda(X)$ can be easily computed during the Newton step and gives an estimate of the residual of the current approximation of the solution.

				 Furthermore, for every $X\in\XWpdpdc(\XWpdpdd)$ the quadratic form of the Hessian in the original coordinates can be expressed as
				 \begin{equation}
					 \sp{b_{XX}\Delta_X}{\Delta_X} = \tr\left({W^{-\frac 12} \Delta W[\Delta_X] W^{-1}  \Delta W[\Delta_X]W^{-\frac 12}} \right).
				 \end{equation}
				 Using the Courant-Fischer theorem twice, see e.\,g.\ \cite{Bel97}, this implies that
				 \begin{align}
					 \tr\left({W^{-\frac 12} \Delta W[\Delta_X] W^{-1}  \Delta W[\Delta_X]W^{-\frac 12}}\right)
	& \ge \frac{1}{\lambda_{\mathrm{max}}(W(X))} \tr \left(\Delta W[\Delta_X] W^{-1}  \Delta W[\Delta_X]\right)\\
	&\ge \frac{1}{\lambda^2_{\mathrm{max}}(W(X))} \tr \left(\Delta W[\Delta_X]  \Delta W[\Delta_X]\right).
				 \end{align}
				 Note that $\|\Delta W[\Delta_X]\|_{F}\neq 0$ for controllable $(A,B)$ and $\Delta_X\neq 0$.
				 Minimizing the left-hand side over all $\Delta_X$ with $\|\Delta_{X}\|_{F}^{2}=1$ yields uniform positivity of the Hessian, since the spectrum of $W(X)$ is bounded.

				 Hence, it follows, see e.\,g.\ \cite{BoyV2004}, that the Newton method is quadratically convergent,  whenever $\lambda(X)<.25$ in some intermediate step. Once this level is reached, the methods stays in the quadratically converging regime. If the condition does not hold, then one has to take a smaller stepsize $(1+\lambda(X))^{-1}\Delta_X$ in order to obtain convergence.

				 \subsubsection{Initialization}

				 Note that for the reformulations of the Newton step we have to assume that the starting value $X_0$ is in the interior of the domain. In this section, we show how to compute an initial point $X_0\in \XWpdpd$, which therefore satisfies  the LMIs $W_c(X_0,\M) \succ  0$ and $W_d(X_0,\M) \succ  0$ for the model $\mathcal M = \{A,B,C,D\}$. Since the reasoning for both the continuous-time case and the discrete-time case are very similar, we first focus on the continuous-time case.

				 We start the optimization from a model $\M$ that is minimal and strictly passive. It then follows that {the solution set of $W_c(X_0,\M) \succ  0$} has an interior point $X_0 \succ 0$ such that
				 \[
					 W_c(X_0,\M) \succ  0, \quad 0 \prec  X_{-}\preceq X_0 \preceq X_{+}.
				 \]
				 To construct such an $X_0$, let $\alpha:=\lambda_{\min}W_c(X_0) > 0$ and $\beta:=\max(\|X_0\|_2,1)>0$. Then, for $0< 2\xi \le \alpha/\beta$, we have the inequality
				 \[  W_c(X_0,\M) \succeq 2\xi \left[\begin{array}{cc} X_0 & 0 \\ 0 & I_m \end{array}\right].
				 \]
				 In order to compute a solution $X_0$ for this LMI, we rewrite it, up to a scaling factor, as
					 \[   W_c(X,\M_\xi) := \left[\begin{array}{cc} -(A+\xi I_n)^{\mathsf{H}}X-X(A+\xi I_n) &   C^{\mathsf{H}}-XB \\    C-B^{\mathsf{H}}X & (D-\xi I_m)^{\mathsf{H}}+ (D-\xi I_m) \end{array}\right] \succeq 0.
					 \]
					 for the modified model $\M_\xi :=\{A+\xi I_n,B,C,D-\xi I_m\}$.
					 The solution set of this shifted LMI can be obtained from the extremal solutions $X_{-}(\xi)$ and $X_{+}(\xi)$ of the Riccati equations for
					 $\M_\xi$.
					 It therefore follows that 
					 \[
						 0 \prec  X_{-} \prec   X_{-}(\xi) \preceq X_{+}(\xi) \prec  X_{+}.
					 \]

					 The reasoning for the discrete-time case is very similar. Starting from a strictly passive and minimal model $\M$, we have the inequality
					 \[
						 W_d(\M) \succeq 2\xi \left[\begin{array}{cc} X_0 & 0 \\ 0 & I_m \end{array}\right],
						 \quad \mathrm{for} \quad 0< 2\xi \le \alpha/\beta=\lambda_{\min}W_d(X_0)/\max(\|X_0\|_2,1).
					 \]
					 In order to compute a solution $X_0$ for this LMI, we rewrite it as an LMI
					 \[
						 W_d(X_0,\M_\xi) := \left[\begin{array}{cc} X_0-A_\xi^{\mathsf{H}}X_0A_\xi &   C_\xi^{\mathsf{H}}-A_\xi^{\mathsf{H}} X_0B_\xi \\    C_\xi-B_\xi^{\mathsf{H}}X_0 A_\xi & D_\xi ^{\mathsf{H}}+ D_\xi -B_\xi^{\mathsf{H}} X_0B_\xi  \end{array}\right] \succeq 0
					 \]
					 for the modified model $\M_\xi :=\{A_\xi,B_\xi,C_\xi,D_\xi \} :=
					 \{A/\sqrt{1-2\xi},B/\sqrt{1-2\xi},C/(1-2\xi),(D-\xi I_m)/(1-2\xi) \}$.
					 The solution set  $X_{-}(\xi)\preceq X_0 \preceq X_{+}(\xi)$ of this scaled LMI is again strictly included in the original solution set.

					 The procedure to find an inner point is thus to choose one of the Riccati solutions $X_{-}(\xi)$ or $X_{+}(\xi)$ of shifted or scaled				 problems, respectively, or some kind of average of both, since they are then guaranteed to be an interior point of the original problem.

					 Another possibility to compute an initial point is to take the \emph{geometric mean} of
					 the minimal and maximal solution of the Riccati equations \eqref{riccatic}, respectively \eqref{riccatid}, denoted by $X_-$ and $X_+$, which is defined by
					 $X_0 = X_- ( X_-^{-1}X_+)^{\frac12}$, see \cite{Moa05}.
					 However, e.\,g., if $X_-$ and $X_+$ are multiples of the identity matrix, then the geometric mean is a convex combination of $X_-$ and $X_+$ and will not necessarily be in the interior.


					 \subsection{Numerical results}\label{sec:numericalresults}
					 We have implemented the steepest ascent method of Subsection~\ref{sec:ascent} and the Newton method introduced in Subsection~\ref{sec:newton}. The
					 software package is written in \texttt{python 3.6}. The code and all the
					 examples can be downloaded under \cite{BanMNV19_code}. 

					 We have performed several experiments to test convergence for the different methods developed in this paper.

					 \begin{example}\label{ex:1}
						 As a prototypical  example consider a randomly generated continuous-time example with $n=30$ and $m=10$, \ie the overall dimension of the matrix $W_c(X)$ is $40\times 40$ and we have a total of $465$ unknowns.
						 \\ As one would expect, the steepest ascent method shows linear convergence behavior, whereas the Newton method has quadratic convergence as soon as one is close enough to the analytic center.

						 Figure~\ref{fig:convergence_plots} shows the convergence behavior using the Newton method. Note, that the barrier function $\det(W(X))$ increases monotonously, whereas the distance of the argument $X$ to the analytic center $X_c$ slightly increases in the linearly convergent phase.
						 The number of steps required in the steepest ascent approach, however, is much  higher than in the Newton approach.
					 \end{example}

					 Also, the initial point computed by the geometric mean approach turns out to be much better in all the practical examples, even though one cannot guarantee positivity in some extreme cases.

					 Note that one has to be extremely careful with the implementation of the algorithm. Without explicitly forcing the intermediate solutions $X_k$ to be Hermitian in finite precision arithmetic, the intermediate Riccati residuals $P_k$ may diverge from the Hermitian subspace.

					 \begin{figure}
						 \centering
						 \begin{subfigure}{0.5\textwidth}
							 \centering
							 \includegraphics[width=\textwidth]{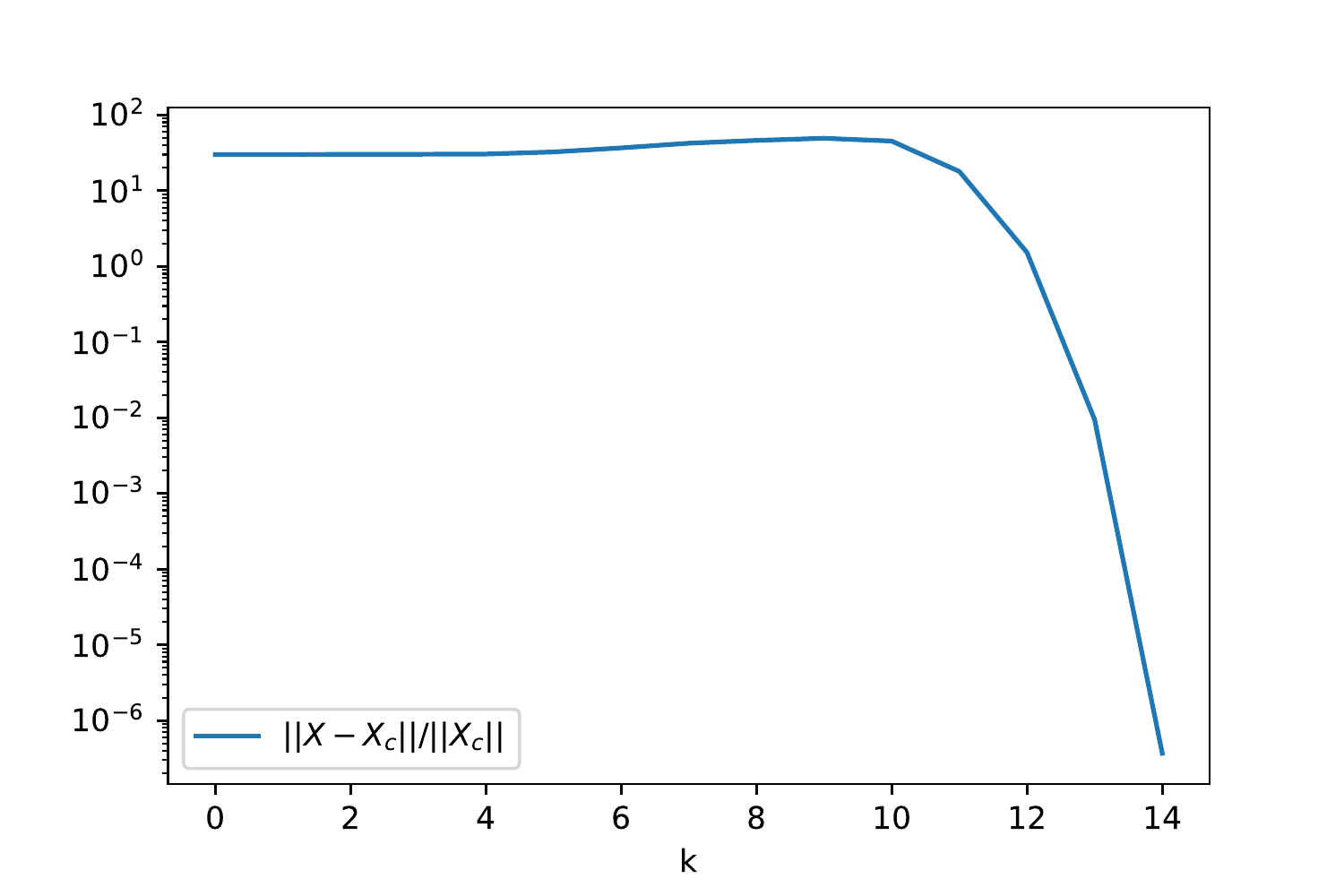}
							 \caption{Convergence of the relative error between the current solution $X_k$ and the analytic center $X_c$\\} 
						 \end{subfigure}%
						 ~
						 \begin{subfigure}{0.5\textwidth}
							 \centering
							 \includegraphics[width=\textwidth]{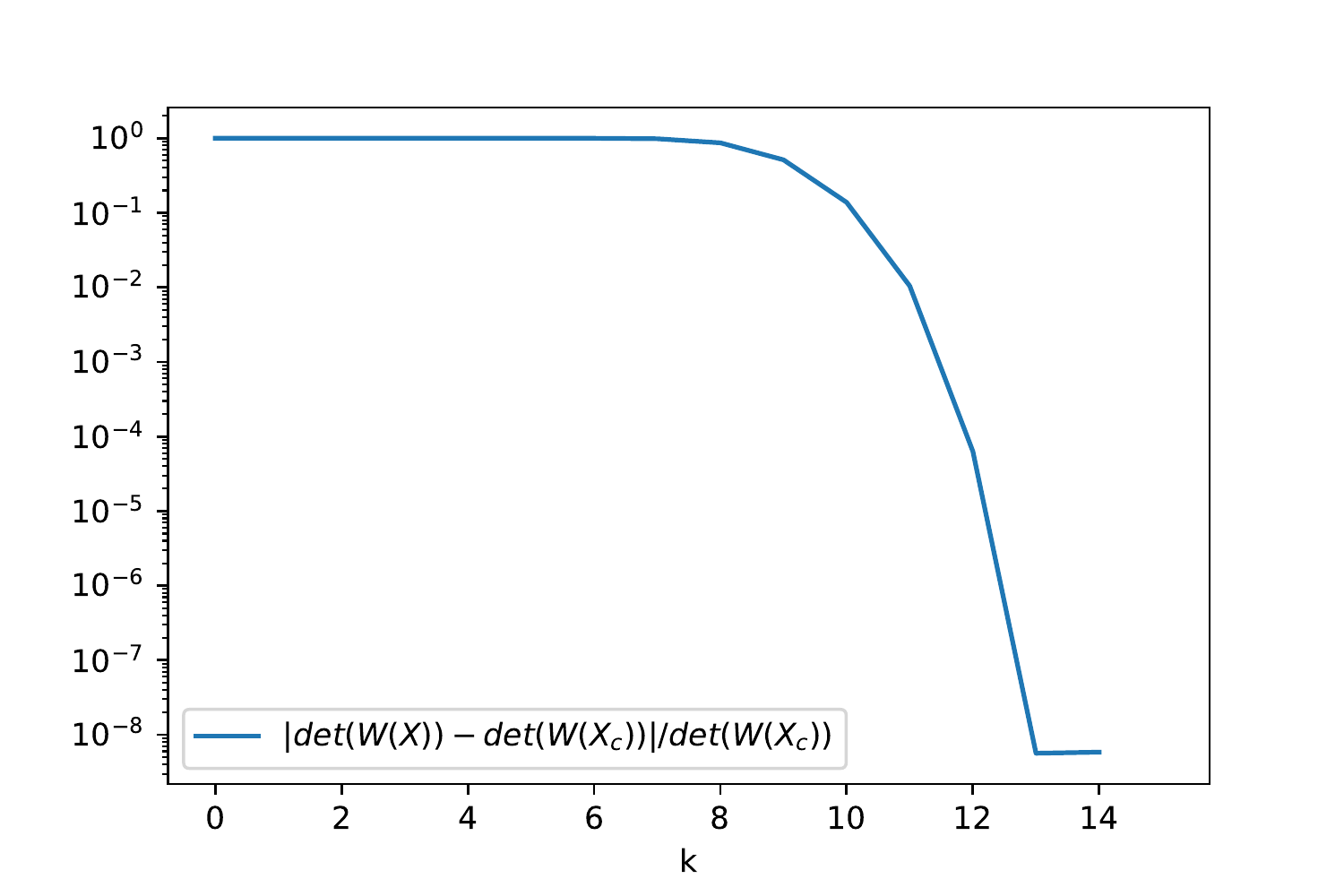}
							 \caption{Convergence of the relative error between the current value of the objective function {$\det(W_c(X_k))$ and the value $\det(W_c(X_c))$} at the analytic center } 
						 \end{subfigure}%
						 \caption{Convergence behavior for the Newton method applied to the example as in Example~\ref{ex:1}}\label{fig:convergence_plots}
					 \end{figure}

					 \section{Computation of bounds for the passivity radius} \label{sec:passivity}
					 Once we have found a solution $X\in\XWpdpdc$, respectively $X\in \XWpdpdd$, we can use this solution to find an estimate of the \emph{passivity radius} of our system, \ie the smallest perturbation $\Delta_{\M}$ to the system coefficients $\M=\{A,B,C,D\}$ that puts the system on the boundary of the set of passive systems, so that an arbitrary small further perturbation makes the system non-passive. In this section we derive {\em lower bounds} \/for the passivity radius in terms of the smallest eigenvalue of a scaled version of the matrices $W_c(X,\M)$ or $W_d(X,\M)$, respectively. Since the analytic center is central to the solution set of the LMI, we choose it for the realization of the transfer function, since then we expect  to maximize a very good lower bound for the passivity radius.

					 \subsection{The continuous-time case}
					 As soon as we fix $X\in \XWpdpdc$, the matrix
					 \begin{equation} \label{hx}
						 W_c(X,\M) = \left[
							 \begin{array}{cc}
								 -A^{\mathsf{H}}X - X\,A & C^{\mathsf{H}}-X\,B \\
								 C-B^{\mathsf{H}}X & D+D^{\mathsf{H}}
							 \end{array}
						 \right]
					 \end{equation}
					 is linear as a function of the coefficients  $A,B,C,D$.
					 When perturbing the coefficients, we thus preserve strict passivity, as long as
					 \begin{align}
& W_c(X,\M+\Delta_\M)\\
&\qquad :=\left[
	\begin{array}{cc}
		-(A+\Delta_A)^{\mathsf{H}}X - X\,(A+\Delta_A) & (C+\Delta_C)^{\mathsf{H}}-X (B+\Delta_B)) \\
		(C+\Delta_C)-(B+\Delta_B)^{\mathsf{H}}X & (D+\Delta_D) +(D+\Delta_D)^{\mathsf{H}}
	\end{array}
					 \right] \\ &\qquad \succ 0.
				 \end{align}
				 We thus suppose that $W_c(X,{\M})\succ  0$ and look for the smallest perturbation $\Delta_{\M}$ to our model $\M$ that makes $\det W_c(X,\M+\Delta_{\M})=0$. To measure the model perturbation, we propose to use the norm of the perturbation of the system pencil
				 \[ \|\Delta_\M \| :=  \left \| \left[\begin{array}{ccc}
							 0 & \Delta_A & \Delta_B \\
							 \Delta_A^{\mathsf{H}} & 0 & \Delta_C^{\mathsf{H}} \\
							 \Delta_B^{\mathsf{H}} & \Delta_C & \Delta_D+\Delta_D^{\mathsf{H}}
							 \end{array}\right] \right \|_2 \approx \left \|\left[\begin{array}{ccc}
							 \Delta_A & \Delta_B \\
							 \Delta_C & \Delta_D
					 \end{array}\right] \right \|_2.
					 \]
					 We have the following lower bound in terms of the smallest eigenvalue $\lambda_{\min}$ of a scaled version of $W_c(X,\M)$.
					 \begin{lemma} \label{lowerbound}
						 The \emph{$X$-passivity radius}, defined for a given $X\in \XWpdpdc$  as
						 \[
							 \rho_\M^c(X):= \inf_{\Delta_\M}\{ \| \Delta_\M \|  | \det W_c(X,\M+\Delta_\M) = 0\},
						 \]
						 satisfies
						 \begin{equation} \label{bound}
							 \lambda_{\min}(Y_c W_c(X,\M) Y_c) \le \rho_\M^c(X),
						 \end{equation}
						 for
						 \[
							 Y_c:= \left[\begin{array}{cc}
									 I_n+X^2 & 0 \\ 0 & I_m
							 \end{array}\right]^{-\frac{1}{2}}\preceq I_{n+m}.
						 \]
					 \end{lemma}
					 \begin{proof}
						 We first note that
						 %
						 \begin{align}
& \det \left[\begin{array}{cccc}
		0 & I_n & X & 0 \\
		I_n & 0 & A+\Delta_A & B+\Delta_B \\
		X & A^{\mathsf{H}}+\Delta_A^{\mathsf{H}} & 0 & C^{\mathsf{H}}+\Delta_C^{\mathsf{H}} \\
		0 & B^{\mathsf{H}}+\Delta_B^{\mathsf{H}} &  C+\Delta_C & R+\Delta_R^{\mathsf{H}}
\end{array}\right] \nonumber \\
& \qquad
=\det  \left[\begin{array}{cc}
0 & I_n \\  I_n & 0  \end{array}\right]
\det W_c(X,\M+\Delta_\M),
\label{embed}
\end{align}
since $W_c(X,\M+\Delta_\M)$ is just the Schur complement with respect to the
leading $2n\times 2n$ matrix.
Here we have set $R:=D+D^{\mathsf{H}}$ and $\Delta_R:= \Delta_D+\Delta_D^{\mathsf{H}}$.

If we introduce the $n\times (n+m)$ matrix $Z_c := \left[\begin{array}{cc} -X & 0 \end{array}\right]$, then \eqref{embed} is equivalent to
\[
\left[\begin{array}{c} Z_c \\ I_{m+n} \end{array}\right]^{\mathsf{H}}
\left[\begin{array}{ccc}
 0 & A+\Delta_A & B+\Delta_B \\
 A^{\mathsf{H}}+\Delta_A^{\mathsf{H}} & 0 & C^{\mathsf{H}}+\Delta_C^{\mathsf{H}} \\
 B^{\mathsf{H}}+\Delta_B^{\mathsf{H}} &   C+\Delta_C & R+\Delta_R
\end{array}\right]  \left[\begin{array}{c} Z_c \\ I_{m+n} \end{array}\right] =
W_c(X, \M+\Delta_\M).
\]
If we replace the matrix $\left[\begin{array}{c} Z_c \\ I_{m+n} \end{array}\right]$ by the matrix  $U_c = \left[\begin{array}{c} Z_c \\ I_{m+n} \end{array}\right]Y_c$ with orthonormal columns, which we can e.\,g.\ obtain from a QR decomposition \cite{GolV96}, 
then we obtain
\begin{align} & U_c^{\mathsf{H}} \left[\begin{array}{ccc}
	 0 & A+\Delta_A & B+\Delta_B \\
	 A^{\mathsf{H}}+\Delta_A^{\mathsf{H}} & 0 & C^{\mathsf{H}}+\Delta_C^{\mathsf{H}} \\
	 B^{\mathsf{H}}+\Delta_B^{\mathsf{H}} &   C+\Delta_C & R+\Delta_R
\end{array}\right] U_c \\
       &\qquad = Y_cW_c(X,\M+\Delta_\M)Y_c \\
       &\qquad=
       Y_c W_c(X,\M) Y_c
	 +U_c^{\mathsf{H}}\left[\begin{array}{ccc} 0 & \Delta_A & \Delta_B \\
			 \Delta_A^{\mathsf{H}} & 0 & \Delta_C^{\mathsf{H}} \\
			 \Delta_B^{\mathsf{H}} & \Delta_C & \Delta_R
	 \end{array}\right]U_c.
		 \end{align}
		 Therefore, the smallest perturbation of the matrix $ Y_c W_c(X,\M) Y_c$ to make
		 $Y_cW_c(X,\M+\Delta_\M)Y_c$ singular must have a 2-norm which is at least
		 as large as $\lambda_{\min}( Y_c W_c(X) Y_c)$, and
		 since the perturbation
		 is a contraction of the proposed one,
		 the lower bound in (\ref{bound}) follows.
	 \end{proof}

	 \subsection{The discrete-time case}
	 In the discrete-time case, for a fixed $X$ the LMI takes the form
	 \begin{equation} \label{hdx}
		 W_d(X) =
		 \left[\begin{array}{cc}
				 -A^{\mathsf{H}}XA +  X & C^{\mathsf{H}} -A^{\mathsf{H}}XB \\
				 C-B^{\mathsf{H}}XA & D+D^{\mathsf{H}}-B^{\mathsf{H}}XB
			 \end{array}
		 \right]
		 \succeq 0,
	 \end{equation}
	 and its perturbed version is
	 \begin{align}
&W_d(X,\M+\Delta_\M) \\
&  \quad :=\left[\begin{array}{cc}
		-(A+\Delta_A)^{\mathsf{H}}X(A+\Delta_A)+ X & (C+\Delta_C)^{\mathsf{H}}- (A+\Delta_A)^{\mathsf{H}}X(B+\Delta_B) \\
		C+\Delta_C-(B+\Delta_B)^{\mathsf{H}}X(A+\Delta_A) & R+\Delta_R -(B+\Delta_B)^{\mathsf{H}}X(B+\Delta_B)
	\end{array}
\right]\\
& \quad  \succ  0,
	 \end{align}
 where again  $R:=D+D^{\mathsf{H}}$ and $\Delta_R:= \Delta_D+\Delta_D^{\mathsf{H}}$.

 Note that, in contrast  to the continuous-time case,  for given $X\in \XWpdpdd$ ,
 $W_d(X,\M+\Delta_{\M})$ is not linear in the perturbations. Nevertheless, we have an analogous bound as in Lemma~\ref{lowerbound} also in the discrete-time case.
 \begin{lemma} \label{lowerboundDT}
	 The \emph{$X$-passivity radius}, defined for a given $X\in \XWpdpdd$  as
	 \[
		 \rho_\M^d(X):= \inf_{\Delta_\M}\{ \| \Delta_\M \|  | \det W_d(X,\M+\Delta_\M) = 0\},
	 \]
	 satisfies
	 \begin{align}
		 \nonumber	& \lambda_{\min}( Y_d \left( W_d(X,\M) -
			 \left[\begin{array}{cc} A^{\mathsf{H}} +I_n \\ B^{\mathsf{H}}  \end{array}\right]\frac{X}{2}\left[\begin{array}{cc} \Delta_A & \Delta_B  \end{array}\right] -
		 \left[\begin{array}{c} \Delta_A^{\mathsf{H}}\\ \Delta_B^{\mathsf{H}} \end{array} \right]\frac{X}{2}\left[\begin{array}{cc} A +I_n & B  \end{array}\right]  \right) Y_d)\\
		 \label{boundd} & \qquad \le \rho^d_\M(X),
	 \end{align}
	 where
	 \[
		 Y_d:= \left[\begin{array}{cc}
				 I_{n+m} + Z_d^{\mathsf{H}}Z_d
			 \end{array}\right]^{-\frac{1}{2}}\preceq I_{n+m}, \ Z_d=-\left[\begin{array}{cc} \frac{X}{2}(A+\Delta_A-I_n) & \frac{X}{2}(B+\Delta_B) \end{array}\right].
		 \]
	 \end{lemma}
	 \begin{proof}
		 We first observe that
		 %
		 \begin{align} \nonumber
& \det \left[\begin{array}{cccc}
0 & I_n & \frac{X}{2}(A+\Delta_A-I_n) & \frac{X}{2}(B+\Delta_B) \\
I_n & 0 & A+\Delta_A+I_n & B+\Delta_B \\
(A^{\mathsf{H}}+\Delta_A^{\mathsf{H}}-I_n)\frac{X}{2} & A^{\mathsf{H}}+\Delta_A^{\mathsf{H}}+I_n & 0 & C^{\mathsf{H}}+\Delta_C^{\mathsf{H}} \\
(B^{\mathsf{H}}+\Delta_B^{\mathsf{H}})\frac{X}{2} & B^{\mathsf{H}}+\Delta_B^{\mathsf{H}} &   C+\Delta_C & R+\Delta_R
\end{array}\right]  \\
			 \label{embedd} & \quad=
			 \det  \left[\begin{array}{cc}
			 0 & I_n \\  I_n & 0  \end{array}\right]
			 \det W_d(X,\M+\Delta_\M),
		 \end{align}
		 since again $W_d(X,\M+\Delta_\M)$ is just the Schur complement with respect to the
		 leading $2n\times 2n$ matrix. Note that this matrix \eqref{embedd} is linear in the perturbation parameters, since $X$ is fixed. Using the definition of the matrix $Z_d$, then from \eqref{embedd}, it follows that we can consider
		 \begin{align}
&  \left[\begin{array}{cc}  Z_d^{\mathsf{H}} & I_{m+n} \end{array}\right]
\left[\begin{array}{ccc}
0 & A+\Delta_A+I_n & B+\Delta_B \\
A^{\mathsf{H}}+\Delta_A^{\mathsf{H}}+I_n & 0 & C^{\mathsf{H}}+\Delta_C^{\mathsf{H}} \\
B^{\mathsf{H}}+\Delta_B^{\mathsf{H}} &   C+\Delta_C & R+\Delta_R
\end{array}\right]  \left[\begin{array}{c} Z_d \\ I_{m+n} \end{array}\right]\\
			     & \qquad =
			     W_d(X,\M+\Delta_\M).
\end{align}
If we replace the matrix $\left[\begin{array}{c} Z_d\\ I_{m+n} \end{array}\right]$ by the matrix
with orthonormal columns   $U_d = \left[\begin{array}{c}  Z_d \\ I_{m+n} \end{array}\right]Y_d$, 
then we have
\[ U_d^{\mathsf{H}}
\left[\begin{array}{ccc}
	0 & A+\Delta_A+I_n & B+\Delta_B \\
	A^{\mathsf{H}}+\Delta_A^{\mathsf{H}}+I_n & 0 & C^{\mathsf{H}}+\Delta_C^{\mathsf{H}} \\
	B^{\mathsf{H}}+\Delta_B^{\mathsf{H}} &   C+\Delta_C & R+\Delta_R
\end{array}\right] U_d = Y_d W_d(X,\M+\Delta_\M)Y_d
\]
from which it follows that
\begin{align}
&Y_d W_d(X,\M+\Delta_\M)Y_d =  U_d^{\mathsf{H}}
\left[\begin{array}{ccc}
0 & \Delta_A & \Delta_B \\
\Delta_A^{\mathsf{H}} & 0 & \Delta_C^{\mathsf{H}} \\
\Delta_B^{\mathsf{H}} &   \Delta_C & \Delta_R
\end{array}\right] U_d \\
&
+ Y_d \left( W_d(X,\M)- \left[\begin{array}{cc} A^{\mathsf{H}} +I_n \\ B^{\mathsf{H}}  \end{array}\right]\frac{X}{2}\left[\begin{array}{cc} \Delta_A & \Delta_B  \end{array}\right] -
\left[\begin{array}{c} \Delta_A^{\mathsf{H}}\\ \Delta_B^{\mathsf{H}} \end{array} \right]\frac{X}{2}\left[\begin{array}{cc} A +I_n & B  \end{array}\right]  \right) Y_d,
\end{align}
and the smallest perturbation of the matrix $ Y_d W_d(X,\M) Y_d$ needed to make
$Y_dW_d(X,\M+\Delta_\M)Y_d$ singular must have a 2-norm which is at least
as large as
\[\lambda_{\min}(  U_d^{\mathsf{H}} \left[\begin{array}{ccc}
	0 & A +I_n & B\\
	A^{\mathsf{H}} +I_n & 0 & C^{\mathsf{H}}\\
	B^{\mathsf{H}} &   C & R
\end{array}\right] U_d) \approx \lambda_{\min}( Y_d W_d(X,\M) Y_d).
\]
Again, since the perturbation
is a contraction of the proposed one, the (approximate) lower bound in (\ref{boundd}) follows.
\end{proof}

\subsection{Examples with analytic solution}
In this subsection, to illustrate the results, we present simple examples of scalar transfer functions ($m=1$) of first degree ($n=1$).

Consider first an asymptotically stable continuous-time system and transfer function  $T(s)=d+\frac{cb}{s-a}$ \ie with $a<0$. Then
\[
W_c(x) = \left[\begin{array}{cc}
-2ax  & c-bx \\ c-bx & 2d \end{array}\right]
\]
and its determinant is $\det (W_c(x)) = -4adx-(c-bx)^2$, which is maximal at the central point $x_a=\frac{c}{b}-\frac{2ad}{b^2}$. We then get
\[
W_c(x_a) = \left[\begin{array}{cc}
4d\frac{a^2}{b^2}-2c\frac{a}{b} & 2d\frac{a}{b} \\ 2d\frac{a}{b} & 2d \end{array}\right]
= \left[\begin{array}{cc} 1 & \frac{a}{b} \\ 0 & 1 \end{array}\right].
\left[\begin{array}{cc}p & 0 \\ 0  & 2d \end{array}\right]
\left[\begin{array}{cc} 1 & 0 \\ \frac{a}{b} & 1\end{array}\right],
\]
with $p=2d\frac{a^2}{b^2}-2c\frac{a}{b}$
which implies that $\det (W_c(x_a))=2d \cdot p$.
For the transfer function
to be strictly passive, it must be asymptotically stable and positive on the imaginary axis and hence also
at $0$ and $\infty$. Thus, we have the conditions
\begin{equation} \label{condCT}  a < 0,\ d > 0, \ \frac{da-cb}{a} > 0.
\end{equation}
The function $\Phi_c(\imath\omega)=2d - \frac{2acb}{a^2+\omega^2}$
is a unimodal function, which reaches its minimum either at $0$ (namely
$\Phi_c(0)=p\frac{b^2}{a^2}$) or at $\infty$ (namely $\Phi_c(\infty)=2d$)
and hence the conditions in \eqref{condCT} are sufficient to check passivity.
Thus, for the  model $\M$,  strict  passivity gets lost when either one of the following happens
\[ d+\delta_d=0, \quad a+\delta_a = 0, \quad
\left[ \begin{array}{cc} c+\delta_c & d+\delta_d \end{array} \right]
\left[ \begin{array}{c} -b-\delta_b \\ a+\delta_a \end{array} \right]=0.
\]
Therefore, it follows that
\[
\rho = \min\left(d, a, \sigma_2 \left[ \begin{array}{ccc} a & b \\ c & d  \end{array} \right]\right)=  \sigma_2 \left[ \begin{array}{ccc} a & b \\ c & d  \end{array} \right]
\]
At the analytic center $x_a$ we have
\[
\det W_c(x_a)=2dp=  4\frac{ad}{b^2}(ad-bc)
\]
and the smallest perturbation of the parameters that makes this determinant go to $0$, yields exactly the same conditions as \eqref{condCT}.
This illustrates that the $X$-passivity radius at the analytic center yields a very good condition for strict passivity of the model.

In the discrete-time case the transfer function is $T(z)=d+\frac{cb}{z-a}$ and for it be asymptotically stable we need $a^2<1$, when we assume the coefficients to be real. Then
\[
W_d(x) =
\begin{bmatrix}
	x- a^2x & c -abx\\
	c - abx & 2d - b^2x
\end{bmatrix}
\]
and the analytic center, where $ \det W_d(x) = (1-a^2)x(2d-b^2x)-(c-abx)^2$ is maximal, is given by $x_a=\frac{d-a^2d+a b c}{b^2}$ with
\[
\det W_d(x_a)=\frac{\left(a^2-1\right) (bc-(a-1)d) (b c-(a+1)d)}{b^2}.
\]
The function $\Phi_d(z) =
\frac{b c}{\frac{1}{z}-a}+\frac{b c}{z-a}+2 d$ will be minimal on the unit circle at $z=1$ or $z=-1$. Thus positivity will be lost, when either $a$ reaches $1$ or $-1$, or $bc -(a-1)d=0$ or $bc -(a+1)d=0$. This is exactly the condition also reflected in the determinant of $W(x_c)$ at the analytic center~$x_a$. This again illustrates that the $X$-passivity radius at the analytic center gives a good bound the passivity radius of the system.

\section{Concluding remarks}
We have derived conditions for the analytic center of the linear matrix inequalities (LMIs)
associated with the passivity of linear continuous-time or discrete-time systems.
We have presented numerical methods to compute these analytic centers with steepest ascent and Newton-like methods and we have presented lower bounds for the passivity radii associated with the LMIs evaluated at the respective analytic center.

\bibliographystyle{plainurl}
\bibliography{BMNV2017}
\appendix

\section{Derivatives of functions of complex matrices}\label{der_complex_functions}
In this appendix we present a precise derivation of the formulas for the differentiation of a matrix function with respect to a complex matrix.
Here we distinguish between complex vector spaces $\mathbb{C}^n$ and the
corresponding real vector space  $\mathbb{R}^n + \imag\mathbb{R}^n$. Both
spaces can be identified by $\mathbf{c}: \mathbb{C}^n \rightarrow \mathbb{R}^n
+ \imag\mathbb{R}^n,\,\mathbf{c}(v) = \Re(v) + \imag \Im({v})$.
For matrix spaces of dimension $m\times n$ we use the usual identification with the vector spaces $ \mathbb{C}^n$ and $\mathbb{R}^n + \imag\mathbb{R}^n$. The space $\mathbb{C}^{n}$ is equipped with the standard scalar product $\langle x,y \rangle_{\mathbb{C}}:= x^{\mathsf{H}}y$.
By $\diffX$ we denote the differentiation in a real vector space, whereas the differentiation of a holomorphic function $g$ is denoted by $g'$. Note that if we write $\mathbf c\circ g(x) = u(x_r + \imag x_i) + \imag v(x_r + \imag x_i)$, then by the Cauchy-Riemann equations, see e.\,g.\ \cite{Fre05}, we have $\mathbf c\circ  g'(x) = \diffX[x_r] u(x_r + \imag x_i)- \imag \diffX[x_i] u(x_r + \imag x_i).
$
Then we have the following result:
\begin{lemma}
Assume that $g:  \mat{n}{n} \rightarrow \mathbb{C}$ is holomorphic. Then $f: \mat[R] nn + \imag\mat[R] nn \rightarrow \mathbb{R}$ defined  by

\begin{equation}
	f(X_r + \imag X_i) : = \Re g(X)
\end{equation}
is differentiable over $\mathbb R$ with
\begin{equation}
	\diffX f(X_r + \imag X_i) =  \Re (\overline{g'(X)}\circ \mathbf{c^{-1}})
\end{equation}
and
\begin{equation}
	\Big\langle \diffX f(X_r + \imag X_i), \Delta \Big\rangle_\mathbb{R}
	=\Re \langle \overline{g'(X)}, \mathbf{c^{-1}}(\Delta) \rangle_\mathbb{C}, \quad \Delta = \Delta_r + \imag \Delta_i.
\end{equation}
\end{lemma}
For the holomorphic function $g(X) = \det(X)$ the following fact is well-known, see e.\,g.~\cite{MagN99} for a proof in the real case, that easily extends to the complex case.
%
\begin{lemma}[Jacobi's formula]
Let $g(X) = \det(X)$ and $X\in\mat nn$. Then  $g'(X)=\adj(X^T)$ and the directional derivative of $g$ in direction $\Delta\in\mat nn$ equals
\begin{equation}
	g'(X)\circ\Delta = \tr(\adj(X)\Delta) = \langle \adj(X)^{\mathsf{H}}, \Delta \rangle_\mathbb{C}.
\end{equation}
\end{lemma}
Applying the chain-rule we finally obtain the differentiation formula, which is used throughout this paper.
\begin{corollary}
Let $f: \mat[R] nn + \imag\mat[R] nn \rightarrow \mathbb{R}$ with $f(X_r +
\imag X_i) = \ln \Re \det(X)$ and $X\in\mat nn$ with $\Re \det(X)>0$. Then
\begin{equation}
	\diffX f(X_r + \imag X_i) = \mathbf c \circ
	\left(\frac{\overline{\det(X)}}{\Re\det(X)}X^{-\mathsf H}\right).
\end{equation}
\end{corollary}

\section{Differences between continuous-time and discrete-time systems}\label{appendix:cayley}
Usually, statements for a continuous linear time-invariant system can be
transformed back and forth to discrete-time systems using some bilinear
transform. However, the equations determining the analytic center in both
cases are cubic in $X$, which suggests that there might not be a one-to-one
correspondence. We have shown that the eigenvalues of the feedback system matrix $A_{F_c}$ at the analytic center lie on the imaginary axis in the continuous-time case, whereas they lie inside the unit disk in the discrete-time setting.  In this appendix we show that it is indeed necessary to consider the continuous-time and discrete-time case separately by showing that the three equations determining the analytic center are not preserved under the usual bilinear transformations.

\subsection{Bilinear transformations}
The bilinear transformation
$s=(z-1)/(z+1)$ maps every  asymptotically  stable continuous-time system
$\{ A_c,B_c,C_c,D_c \}$ to a corresponding asymptotically stable discrete-time system
$\{ A_d,B_d,C_d,D_d \}$.
For some $Q_c, Q_d\in\mat nn$ and $R_c = D_c + D_c^{\mathsf{H}},\,R_d = D_d + D_d^{\mathsf{H}}$ set
\begin{equation}
W_c := \left[ \begin{array}{cc} Q_c & C_c^{\mathsf{H}} \\ C_c & R_c \end{array} \right], \qquad W_d := \left[ \begin{array}{cc} Q_d & C_d^{\mathsf{H}} \\ C_d & R_d \end{array} \right].
\end{equation}
Then, starting from a continuous-time system $\{ A_c,B_c,C_c,D_c \}$ we obtain a transformed discrete-time system by setting
\begin{align}
\begin{split}
	\label{bilin}
	A_d & :=  (A_c-I)^{-1}(I+A_c) \\
	B_d & :=  \sqrt{2}(A_c-I)^{-1}B_c \\
	T_c &:=
	\left[ \begin{array}{cc} \sqrt{2}(I-A_c)^{-1} & (I-A_c)^{-1}B_c
	\\ 0 & I \end{array}\right],\\
	W_d & :=  T^{\mathsf{H}}_cW_cT_c,
\end{split}
		\end{align}
		where $C_d, D_d,$ and $Q_d$ are obtained from $W_d$.
		Vice versa, starting from a discrete-time system $\{ A_d,B_d,C_d,D_d \}$ and using the inverse transformation $z=(1+s)/(1-s)$ we obtain a continuous-time system by setting
		\begin{align}
			A_c & :=  (A_d+I)^{-1}(A_d-I) \nonumber\\
			B_c & :=  \sqrt{2}(A_d+I)^{-1}B_d \nonumber \\
			T_d &:=
			\left[ \begin{array}{cc} \sqrt{2}(I+A_d)^{-1} & -(I+A_d)^{-1}B_d
			\\ 0 & I \end{array}\right],\\
			W_c & :=  T^{\mathsf{H}}_dW_dT_d\nonumber.
			\label{bilininv}
		\end{align}
		Note that $(I-A_d) =  2(I-A_c)^{-1}$.

		Bilinear transformations preserve asymptotic stability, and they also relate the domains of the continuous-time and discrete-time linear matrix
		inequalities. To see this, we express the two LMIs as
		\begin{align}
			W_c(X) &:= \left[ \begin{array}{cc} Q_c & C_c^{\mathsf{H}}\\ C_c & R_c \end{array} \right]
			- \left[\begin{array}{cc} A^{\mathsf{H}}_c & I \\ B^{\mathsf{H}}_c & 0 \end{array} \right]
			\left[\begin{array}{cc} 0 & X_c \\ X_c & 0 \end{array} \right]
			\left[\begin{array}{cc} A_c & B_c \\ I & 0 \end{array} \right]\succeq0,\\
			W_d(X) &:= \left[ \begin{array}{cc} Q_d & C_d^{\mathsf{H}} \\ C_d & R_d \end{array} \right]
			- \left[\begin{array}{cc} A^{\mathsf{H}}_d & I \\ B^{\mathsf{H}}_d & 0 \end{array} \right]
			\left[\begin{array}{cc} X_d & 0 \\ 0 & -X_d \end{array} \right]
			\left[\begin{array}{cc} A_d & B_d \\ I & 0 \end{array} \right]\succeq0,
		\end{align}
		respectively. Since
		\[
			\left[\begin{array}{cc} 0 & X_c \\ X_c & 0 \end{array} \right] =
			\left[\begin{array}{cc} I & I \\ I & -I \end{array} \right]
			\left[\begin{array}{cc} \frac{X_c}{2} & 0 \\
			0 & -\frac{X_c}{2} \end{array} \right]
			\left[\begin{array}{cc} I & I \\ I & -I \end{array} \right],
		\]
		we can also express $W_c(X_c)$ as
		\[
			W_c(X_c)=	\left[ \begin{array}{cc} Q_c & C_c^{\mathsf{H}} \\ C_c & R_c \end{array} \right]
			- \left[\begin{array}{cc} A_c+I & B_c \\ A_c-I & B_c \end{array} \right]^{\mathsf{H}}
			\left[\begin{array}{cc} \frac{X_c}{2} & 0 \\
			0 & -\frac{X_c}{2} \end{array} \right]
			\left[\begin{array}{cc} A_c+I & B_c \\ A_c-I & B_c \end{array} \right].
		\]
		Applying the congruence transformation $T_c$ defined in (\ref{bilin}), then
			\[ T^{\mathsf{H}}_cW_c(X_c)T_c= \left[ \begin{array}{cc} Q_d & C^{\mathsf{H}}_d \\
				C_d & R_d \end{array} \right]
				- \left[\begin{array}{cc} A^{\mathsf{H}}_d & I \\ B^{\mathsf{H}}_d & 0 \end{array} \right]
				\left[\begin{array}{cc} X_d & 0 \\ 0 & -X_d \end{array} \right]
				\left[\begin{array}{cc} A_d & B_d \\ I & 0 \end{array} \right],
			\]
			with $A_d, B_d, C_d, D_d$ and $Q_d$ defined as in
			(\ref{bilin}). This shows that maximizing $\det
			W_d(X_d)$ over $X_d$ and maximizing $\det W_c(X_c)$ over $X_c$ is
			equivalent. Thus, the  respective solutions at the
			continuous-time and discrete-time analytic center coincide, \ie $X_d=X_c$.
			The bilinear transformation also
			preserves the solution of the Riccati equation as well as the domain of
			the linear matrix inequality.
			For the 	transformation of the matrices $C_c$, $D_c$, and $Q_c$ we obtain
			\begin{equation}\label{eq:transform_weights}
				\begin{bmatrix}
					Q_d & C_d^{\mathsf{H}}\\
					C_d & D_d + D_d^{\mathsf{H}}
				\end{bmatrix}=
				T_c^{\mathsf{H}}
				\begin{bmatrix}
					Q_c & C_c^{\mathsf{H}}\\
					C_c & D_c + D_c^{\mathsf{H}}
				\end{bmatrix}T_c=
				T_c^{\mathsf{H}}
				\begin{bmatrix}
					\sqrt2 Q_c (I-A_c)^{-1} & Q_c (I-A_c)^{-1} B_c + C_c^{\mathsf{H}}\\
					\sqrt2 C_c (I-A_c)^{-1} & C_c (I-A_c)^{-1}B_c + D_c + D_c^{\mathsf{H}}
				\end{bmatrix}
			\end{equation}
			where the $(1,1)$, $(1,2)$, $(2,2)$ blocks are given by
			\begin{align}
&	2 (I-A_c)^{-\mathsf{H}}Q_c (I-A_c)^{-1},\\
& \sqrt2 (I-A_c)^{-\mathsf{H}}Q_c (I-A_c)^{-1} B_c +
\sqrt2
(I-A_c)^{-\mathsf{H}}C_c^{\mathsf{H}},\\
&	
(I-A_c)^{-1}B_c +B_c^{\mathsf{H}}(I-A_c)^{-\mathsf{H}}Q_c(I-A_c)^{-1} B_c +
B_c^{\mathsf{H}}(I-A_c)^{-\mathsf{H}}C_c^{\mathsf{H}} +
D_c + D_c^{\mathsf{H}},
			\end{align}
			respectively.	

			The transfer function also does not change, provided that one rephrases it in terms of the new variable, \ie $\Phi_d(z)=\Phi_c(s)$. This can be seen as follows. Let us replace the variable $z$ of the system matrix $S_{W_d}(z)$ by $(1+s)/(1-s)$ and then scale the first
			block row and block column by $(1-s)$ and transform the second block row and block column by the upper triangular congruence transformation $T_d$, which does not change the transfer function, then we obtain
				\begin{align} \left[ \begin{array}{c|c} (1-s)I_n & 0  \\ \hline 0 & T_d^{\mathsf{H}} \end{array} \right]
&	\left[ \begin{array}{c|cc} 0 & A_d-zI_n & B_d \\ \hline
zA^{\mathsf{H}}_d -I_n & Q_d  & C^{\mathsf{H}}_d  \\  zB^{\mathsf{H}}_d & C_d  & R_d   \end{array} \right] \left[ \begin{array}{c|c} (1-s)I_n & 0 \\ \hline 0 & T_d \end{array} \right]\\
& \qquad =
\left[ \begin{array}{c|cc} 0 & A_c-sI_n & B_c \\ \hline
A^{\mathsf{H}}_c +sI_n & Q_c  & C^{\mathsf{H}}_c  \\  B^{\mathsf{H}}_c & C_c  & R_c   \end{array} \right].
\end{align}

\subsection{Transformation of the deflating subspaces}
Following \cite{ByeMMX08} we consider the pencils
\begin{equation}\label{eq:evepenc}
	s\E_c-\A_c  =
	\begin{bmatrix}
		0                 & -sI+A_{c} & B_{c}     \\
		sI+ A_{c}^{\mathsf{H}}         &   Q_c   & C_{c}^{\mathsf{H}}    \\
		B_{c}^{\mathsf{H}}             & C_{c}   & D_{c} +D_{c}^{\mathsf{H}}
	\end{bmatrix}  
\end{equation}
corresponding to the continuous-time case and 
\begin{equation}\label{eq:palpenc}
	z\A_d^{\mathsf{H}}-\A_d =
	\begin{bmatrix}
		0                   & zI-A_{d}          & -B_{d}            \\
		zA_{d}^{\mathsf{H}}-I            & (z-1)Q_d        & (z-1)  C_{d}^{\mathsf{H}}      \\
		zB_{d}^{\mathsf{H}}            & (z-1)C_{d}     & (z-1)(D_{d} +D_{d}^{\mathsf{H}})
	\end{bmatrix}
\end{equation}
corresponding to the discrete-time case.

If $X$ is solution of $\Ricc_d(X)=-Q_d$, then there is a deflating subspace of
the form
\begin{align}
& \begin{bmatrix}
	0                   & A_d - zI          & B_d            \\
	I -zA_d^{\mathsf{H}}            & (z-1)Q_d        & (z-1)C_d^{\mathsf{H}}         \\
	-zB_d^{\mathsf{H}}            & (z-1) C_d    & (z-1)(D_d +D_d^{\mathsf{H}})
\end{bmatrix}
\begin{bmatrix}
	-X\left(I - A_d + B_dF_d \right)\\I\\
	-F_d
\end{bmatrix}\\
& \qquad =
\begin{bmatrix}
	I\\
	(I-A_d^{\mathsf{H}})X\\
	-B_d^{\mathsf{H}}X
\end{bmatrix}
\left( A_d - B_dF_d - zI\right).
\end{align}
Applying a generalized bilinear transformation to the pencil $s\E_c -\A_c$ gives
\begin{align}
& z \hat{\A}_d - \hat{\A}^{\mathsf{H}}_d := z(\E_c-\A_c) - (-\E_c-\A_c) \\
& \qquad =
\begin{bmatrix}
	0                   	& z(A_{c}-I) - (I+A_{c})          & zB_{c} - B_{c}            \\
	-z(I +A_{c}^{\mathsf{H}}) - (A_{c}^{\mathsf{H}} - I)  & (z-1)Q_c       & (z-1)C_{c}^{\mathsf{H}}        \\
	-zB_{c}^{\mathsf{H}} - B_{c}^{\mathsf{H}}            & (z-1)  C_{c}    & (z-1)(D_{c} +D_{c}^{\mathsf{H}})
\end{bmatrix},
\end{align}
and then performing the bilinear transform from the previous section on the last two
block columns and rows, we obtain the new  pencil
\begin{equation}
	z\check{\A_{d}}- \check{\A_{d}}^{\mathsf{H}} :=
	\begin{bmatrix}
		\frac1{\sqrt2}I & 0 \\
		0 & T_c^{\mathsf{H}}
	\end{bmatrix}
	\left(
	z \hat{\A_{d}} - \hat{\A_{d}}^{\mathsf{H}}\right)
	\begin{bmatrix}
		\frac1{\sqrt2} I & 0 \\
		0 & T_c
	\end{bmatrix}
	=
	\begin{bmatrix}
		0                   & A_d - zI          & B_d            \\
		I -zA_d^{\mathsf{H}}            & (z-1)Q_d        & (z-1)C_d^{\mathsf{H}}      \\
		-zB_d^{\mathsf{H}}            & (z-1)C_d      & (z-1)(D_d +D_d^{\mathsf{H}})
	\end{bmatrix}.
\end{equation}
If, conversely, there is a continuous-time solution $X$ of $\Ricc_c(X)=-Q_c$, we
have the deflating subspace
\begin{equation}
	\begin{bmatrix}
		0                 & -sI+A_{c} & B_{c}     \\
		sI+ A_{c}^{\mathsf{H}}         & Q_{c}     & C_{c}^{\mathsf{H}}     \\
		B_{c}^{\mathsf{H}}             &  C_{c}  & D_{c} +D_{c}^{\mathsf{H}}
	\end{bmatrix}
	\begin{bmatrix}
		-X\\I\\
		-F_c
	\end{bmatrix}
	=
	\begin{bmatrix}
		I\\
		X\\
		0
	\end{bmatrix}
	\left( A_{c}- B_{c}F_c - sI\right).
\end{equation}
Then, using the same transformation we obtain
\begin{align}
& \left(z\check{\A_{d}}- \check{\A_{d}}^{\mathsf{H}}\right)
\begin{bmatrix}
	\sqrt2 I & 0 \\
	0 & T_c^{-1}
\end{bmatrix}
\begin{bmatrix}
	-X\\I\\
	-F_c
\end{bmatrix}
(I-A_{c}+B_{c}F_c)^{-1}\\
& \qquad=
\begin{bmatrix}
	\frac1{\sqrt2}I & 0 \\
	0 & T_c^{\mathsf{H}}
\end{bmatrix}
\begin{bmatrix}
	I\\
	X\\
	0
\end{bmatrix}
\left( z(-I + A_{c}- B_{c}F_c) - (I+ A_{c} - B_{c}F_c)\right)(I-A_{c}+B_{c}F_c)^{-1},
\end{align}
which is equivalent to
\begin{equation}
	\left(z\check{\A_{d}}- \check{\A_{d}}^{\mathsf{H}}\right)
	\begin{bmatrix}
		- X (I-A_{F_d})\\
		I\\
		-\sqrt2 F_c(I-A_{c}+B_{c}F_c)^{-1}
	\end{bmatrix}
	=
	\begin{bmatrix}
		I\\
		(I-A_d)^{\mathsf{H}}X\\
		-B_d^{\mathsf{H}}X
	\end{bmatrix}
	\left(A_{F_d} -zI\right),
\end{equation}
where $A_{F_d}$ denotes the bilinear transform of the matrix $A_{F_c} = A_{c}-B_{c}F_c$.
Thus, the transformed feedback matrix $F_d$ can be defined by
\begin{equation}\label{eq:feedback_cayley}
	 F_d := \sqrt2
	F_c(I-A_{c}+B_{c}F_c)^{-1}.
\end{equation}
%

It needs to be analyzed how the Riccati operator $P_c(X)$ is transformed for a fixed $X$. Clearly, then $X$ fulfills the Riccati equation $\Ricc_c(X) = -Q_c$ with $Q_c:= - P_c(X)$. From the equations $\Ricc_c(X) = -Q_c$ and $\Ricc_d(X) = -Q_d$ one would then expect, that $P_d(X) = - Q_d$. However,
we have the relation
\begin{equation}\label{eq:relationpcpd}
	\begin{bmatrix}
		P_d & 0 \\
		0 & R_d - B_d^{\mathsf{H}}XB_d
	\end{bmatrix}
	=
	\left(
		T_P
	\right)^{\mathsf{H}}
	\begin{bmatrix}
		P_c & 0 \\
		0 & R_c
	\end{bmatrix}
	\underbrace{
		\begin{bmatrix}
			I & 0 \\
			F_c & I
		\end{bmatrix}
		\begin{bmatrix}
			\sqrt{2}(I-A_c)^{-1} & (I-A_c)^{-1}B_c
			\\ 0 & I
		\end{bmatrix}
		\begin{bmatrix}
			I & 0 \\
			-F_d & I
	\end{bmatrix}}_{T_P:=},
\end{equation}
where we compute \begin{equation}
	T_P =
	\begin{bmatrix}
		\sqrt2 (I - A_{c}+B_{c}F_{c})^{-1}& (I - A_{c})^{-1}B_{c}\\
		0 & I + F_{c}(I - A_{c})^{-1}B_{c}
	\end{bmatrix},
\end{equation}
and used that $\sqrt2 F_{c}(I - A_{c})^{-1} - F_d -F_{c}(I - A_{c})^{-1}B_{c}F_d=0$.
We thus obtain that
\begin{equation}
	P_d = 2 (I - A_{c}+B_{c}F_{c})^{-\mathsf{H}}P_c (I - A_{c}+B_{c}F_{c})^{-1},
\end{equation}
which, by considering that $P_c\succ 0$ and equation \eqref{eq:transform_weights}, only coincides with $-Q_d$ if $F_{c}=0$.
Thus we have shown, that if we enforce a feedback, that keeps the feedback
system matrix $A_{F_{d}}$ on the unit circle, then the transformed residual
of
the Riccati operator $P_c$ does not correspond to the discrete-time residual
$P_d$.
In other words, since relation \eqref{eq:relationpcpd} has to hold, the
transformation of the feedback \eqref{eq:feedback_cayley} cannot be true, and
thus the discrete-time feedback system matrix $A_{F_{d}}$ does not lie on the
unit circle. Indeed, as mentioned before, the eigenvalues lie strictly inside
the unit circle.

			\end{document}